\documentclass[smallextended,referee,envcountsect]{svjour3}
\usepackage [latin1]{inputenc}
\usepackage{amsmath,amssymb}
\usepackage{marvosym}
\usepackage{graphicx}
\usepackage{subcaption}
\usepackage{lipsum}
\usepackage{epstopdf}
\usepackage[numbers,sort&compress]{natbib}
\usepackage[colorlinks,linkcolor=blue,urlcolor=blue,citecolor=blue]{hyperref}

\usepackage{xcolor}

\smartqed
\usepackage{graphicx}
\journalname{}

\usepackage{fancyhdr}
\usepackage{booktabs}
	
\usepackage{ntheorem}
\theoremheaderfont{\bfseries\upshape}
\theorembodyfont{\upshape}
\renewtheorem{remark}{\it Remark}[section]
\renewtheorem{example}{Example}[section]
\renewtheorem{property}{Property}[section]
\renewtheorem{lemma}{Lemma}[section]
\newtheorem{assumption}{Assumption}[section]

\begin{document}

\title{Inertial dynamical systems and accelerated algorithms with implicit Hessian-driven damping for nonconvex optimization}

\author{Zeying Gao$^{1}$\and Xiangkai Sun$^{1}$\and Liang He$^{1,2}$}

\institute{Zeying Gao \at {\small gaozy\_6@163.com} \\
	\\Xiangkai Sun  (\Letter) \at{\small sunxk@ctbu.edu.cn } \\
          \\Liang He \at{\small liangheee@126.com}\\\\
               $^{1}$Chongqing Key Laboratory of Statistical Intelligent Computing and Monitoring, School of Mathematics and Statistics,
 Chongqing Technology and Business University,
Chongqing 400067, China.\\
\\
$^{2}$Department of Mathematics, Sichuan University, Chengdu 610064, Sichuan, China
}

\date{Received:   / Accepted:  }

\maketitle

\begin{abstract}
 This paper is devoted to the investigation of inertial dynamical systems with implicit Hessian-driven damping for strongly quasiconvex optimization which is a special class of nonconvex optimization problems. We first establish exponential convergence rate properties for this system without requiring  Lipschitz continuity of the gradient on the function. Then, we obtain an inertial accelerated algorithm for minimizing strongly quasiconvex functions through natural explicit time   discretization of the dynamical system. Meanwhile, we consider an exogenous additive perturbation term to this dynamical system and obtain the corresponding algorithm. By utilizing the Lyapunov method, we establish convergence rates of iterative sequences and their function values. Furthermore, we conduct numerical experiments to illustrate the theoretical results.
\end{abstract}


\section{Introduction}
Let $\mathbb{R}^n$ be an $n$-dimensional Euclidean space and let $f: \mathbb{R}^n \rightarrow \mathbb{R}$ be a   continuously differentiable function. Consider the following optimization problem:
\begin{eqnarray}\label{min.f}
\min_{x \in \mathbb{R}^n} f(x).
\end{eqnarray}
Inertial dynamics method, as one of the powerful approaches for solving problem \eqref{min.f},
has attracted an increasing interest by many researchers from several different
perspectives.

In the case that the objective function $f$ is convex, Polyak \cite{polyak1964some} initially proposed the Heavy Ball method with a friction system:
\begin{eqnarray} \label{chzh}
\ddot{x}(t) + \alpha\dot{x}(t) + \nabla f (x(t)) = 0,
\end{eqnarray}
where $\alpha > 0$ is a constant damping coefficient. The Euler-forward discretization version of system (\ref{chzh}) yields the well-known Heavy Ball method given by
 \begin{align} \label{heavyball}
		\begin{array}{ll}
				x_{k+1} = x_{k} + \alpha (x_{k} - x_{k - 1}) - \beta \nabla f(x_{k}),
	\end{array}
		\tag{HBM}
\end{align}
where   $\alpha > 0$  is the momentum coefficient and $\beta > 0$ is step size. Unfortunately, in the case of a general convex function $f$, \eqref{heavyball}  only provides
a   convergence rate of values of order $\mathcal{O}(\frac{1}{k})$ as $k\rightarrow+\infty$.
Furthermore, \eqref{heavyball}  may lead to oscillatory behaviors in poorly conditioned or flat regions. In the quest for a faster theoretical convergence rate as well as greater stability,   Su et al. \cite{su2016differential} proposed the following dynamical system with an asymptotic vanishing damping coefficient $\frac{\alpha}{t}$:
\begin{eqnarray} \label{vanish}
	\ddot{x}(t) + \frac{\alpha}{t} \dot{x}(t) + \nabla f (x(t)) = 0,
\end{eqnarray}
where   $\alpha$ is a positive constant.
 When $\alpha \geq 3$,   they showed that  the convergence rate of the objective function value is $\mathcal{O}(\frac{1}{t^2})$ as $t$ approaches infinity. Furthermore, for $\alpha \leq 3$, Attouch et al. \cite{attc2019} derived  a convergence rate of $\mathcal{O}(t^{-\frac{2}{3}\alpha})$ for the objective function value along the trajectory generated by the system \eqref{vanish}.  Specifically, for $\alpha=3$, Su et al. \cite{su2016differential} demonstrated that the system \eqref{vanish} is a continuous  scheme  of Nesterov's accelerated gradient method \cite{nesterov1983method} as follows:
\begin{align} \label{NAG}
		x_{k+1} = x_{k} +  \alpha (x_{k} - x_{k - 1})- \beta \nabla f(x_{k} +  \alpha (x_{k} - x_{k - 1})).
	\tag{NAG}
\end{align}
 Note that,  when compared with \eqref{heavyball},  the gradient in \eqref{NAG} is evaluated at an extrapolated term of the
form $x_{k} +  \alpha (x_{k} - x_{k - 1})$. Moreover, \eqref{NAG} achieves the convergence rate of  values $\mathcal{O}(\frac{1}{k^2})$.
  For recent results on the asymptotic properties of the continuous and discrete dynamics in the critical case $\alpha = 3$, we refer the reader to \cite{attc2019,attouch2016fast,ai3, bot3, mulobj3}. 

Additionally, Attouch et al. \cite{attouch2018fast} also considered the perturbed version of \eqref{vanish}, that is,
\begin{eqnarray*} \label{vanish,raod}
	\ddot{x}(t) + \frac{\alpha}{t} \dot{x}(t) + \nabla f (x(t)) = \epsilon(t),
\end{eqnarray*}
where $\alpha \geq 3$ and $\epsilon: [t_0, + \infty) \rightarrow \mathbb{R}^{n}$ is an external perturbation function satisfying $ \int_{t_{0}}^{+\infty} t \| \epsilon(t)\| dt < +\infty$.   They showed that the convergence rate of the objective function value is $\mathcal{O}(\frac{1}{t^2})$ as $t $ approaches infinity, which complemented the rate obtained by Su et al. \cite{su2016differential} in the unperturbed case $\epsilon(t)=0$. Over the past few years, for solving problem \eqref{min.f}, a wide variety of works devoted to  the  system \eqref{vanish} and  its generalizations  have been  obtained from both continuous and discrete time perspectives. We refer the readers to some recent papers \cite{jems,van2017,van2018,tikhonov,s23jde}.

Considering another development on damping terms  and reducing the oscillations that occur during iterations of the Heavy Ball method, Alvarez et al.  \cite{alvarez2002second} proposed the following dynamical system with explicit Hessian-driven damping term $ \beta \nabla^2 f(x(t))\dot{x}(t)$:
\begin{eqnarray} \label{hess}
\ddot{x}(t) + \alpha\dot{x}(t) + \beta \nabla^2 f(x(t))\dot{x}(t) + \nabla f (x(t)) = 0,
\end{eqnarray}
where $\alpha > 0$ and $\beta > 0$. They also established convergence properties of the trajectory generated by system \eqref{hess} under different assumptions on $f$. Subsequently, Attouch et al. \cite{attouch2016fast} introduced the following  dynamical system with asymptotic vanishing damping and Hessian-driven damping terms:
\begin{eqnarray}\label{ga123}
	\ddot{x}(t) + \frac{\alpha}{t}\dot{x}(t) + \beta \nabla^2 f(x(t))\dot{x}(t) + \nabla f (x(t)) = 0,
\end{eqnarray}
where $\alpha > 0$ and $\beta > 0$. In case $\alpha > 3$, they obtained the weak convergence of the trajectory to a minimizer of problem \eqref{min.f} and an $o(\frac{1}{t^2})$ convergence rate for the objective function value along the trajectory. Nowadays, a series of studies on inertial dynamical systems and accelerated algorithms with Hessian-driven damping have garnered considerable attention,  see \cite{firstorder, attouch2021effect, forback, aujo23, tikhonovhes, newton, he25, c25, kl, wangtx, zhong24}.

 It is worth noting that the Heavy Ball method and its variants cannot yield the Nesterov's algorithm through straightforward implicit/explicit time discretization of dynamical systems with explicit Hessian-driven damping term (e.g., systems \eqref{hess} and \eqref{ga123}). In order to overcome this limitation, Alecsa et al. \cite{alecsa2021extension} proposed the following inertial dynamical system with an implicit Hessian-driven damping term: 
\begin{eqnarray} \label{vanish,imhes}
	\ddot{x}(t) + \frac{\alpha}{t} \dot{x}(t) + \nabla f \left( x(t)+\left( \gamma+\frac{\beta}{t} \right) \dot{x}(t) \right) = 0,
\end{eqnarray}
where $\alpha > 0$,  $\beta \in \mathbb{R}$ and $\gamma \geq 0$.  Note that  by applying the Taylor expansion to $\nabla f $ at $x(t)$, one has
$
	\nabla f \left( x(t)+\left( \gamma+\frac{\beta}{t} \right) \dot{x}(t) \right) \approx \nabla f(x(t)) + \left( \gamma+\frac{\beta}{t} \right) \nabla^2 f(x(t)) \dot{x}(t).
$
Thus, $\nabla f \left( x(t)+\left( \gamma+\frac{\beta}{t} \right) \dot{x}(t) \right)$ is referred to as  the implicit Hessian-driven damping term. 
Although there exists a structural similarity between the systems \eqref{ga123} and \eqref{vanish,imhes}, the trajectory generated by system \eqref{vanish,imhes} achieved better convergence performance. Further, by explicitly discretizing system  \eqref{vanish,imhes} in time, a gradient-type algorithm can be directly obtained, which further shows the advantages of the introduction of implicit Hessian-driven damping terms. However, in contrast to systems \eqref{hess} and \eqref{ga123}, the system \eqref{vanish,imhes} and its variants have been less explored in the literature. This is also the first motivation for the investigation of inertial dynamical system with implicit Hessian-driven damping term in this paper.

On the other hand, we observe that both continuous-time dynamical systems and discrete-time accelerated methods mentioned above are designed for problem \eqref{min.f} under convexity conditions. From a dynamical system perspective, it appears that there exists a few papers in the literature devoted to the study of algorithms for problem \eqref{min.f} with a nonconvex objective function. Among them, a constructive strategy involves the study of generalized convex functions, such as quasiconvex functions \cite{cambini2009generalized, quas16, prom,  hadjisavvas2006handbook, qusc19, convexbook}. However, the convergence was established only to critical points under a quasiconvexity assumption, see \cite{quasiconvex2009, quasiconvex2021}. In order to avoid such behaviors, some researchers have focused on the study of problem \eqref{min.f} with the strongly quasiconvex function which was initially introduced by Polyak \cite{poracka1971existence}. Note that the class of strongly quasiconvex functions not only preserves many favorable properties of convex functions but also encompasses some significant nonconvex functions \cite{sofar, korablev1989relaxation, lara2022strongly, nam2024strong}. Consequently, strongly quasiconvex functions are of great value in both theoretical analysis and practical applications. Some important steps in this direction have been made in \cite{hadjisavvas2025heavy, lara2024characterizations}. More precisely, for problem \eqref{min.f} with strongly quasiconvex objective functions, Lara et al. \cite{lara2024characterizations} established exponential convergence rates for system \eqref{chzh} without requiring $\nabla f$ to be Lipschitz continuous. Via explicit time   discretization of the system \eqref{chzh}, they also proposed a Heavy Ball accelerated algorithm which enjoys a linear convergence rate. Further, Hadjisavvas et al. \cite{hadjisavvas2025heavy} established exponential convergence rates for system \eqref{hess} in the context of strongly quasiconvex functions. Through temporal discretization of the system \eqref{hess}, they obtained a Heavy Ball method   \eqref{heavyball} with Hessian correction:
\begin{align}\label{heavyballhess}
	\left\{
	\begin{array}{ll}
		y_{k} = x_{k} + \alpha (x_{k} - x_{k - 1}) - \theta \left(\nabla f(x_k) -
		\nabla  f(x_{k-1})\right), \\ [2mm]
		x_{k+1} = y_{k} - \beta \nabla f(x_{k}),
	\end{array}
	\right.
	\tag{HBM-H}
\end{align}
where $\alpha \in [0,1]$, $\theta \geq 0$, $\beta > 0$.   When $\nabla f(x_{k})$ in \eqref{heavyballhess} is replaced by $\nabla f(y_{k})$,   they also  obtained a Nesterov's accelerated gradient method   \eqref{NAG} with Hessian correction:
\begin{align}\label{NAGhess}
	\left\{
	\begin{array}{ll}
		y_{k} = x_{k} + \alpha (x_{k} - x_{k - 1}) - \theta \left(\nabla f(x_k) -
		\nabla  f(x_{k-1})\right), \\ [2mm]
		x_{k+1} = y_{k} - \beta \nabla f(y_{k}).
	\end{array}
	\right.
	\tag{NAG-H}
\end{align}
 Then, they established the linear convergence rate of algorithms \eqref{heavyballhess} and \eqref{NAGhess}, which further enriched the applications of strongly quasiconvex functions in nonconvex optimization problems.

Motivated by the works reported in \cite{alecsa2021extension, hadjisavvas2025heavy, lara2024characterizations}, in this paper, we further consider the following dynamical system with implicit Hessian-driven damping:
\begin{align} \label{system1}
	\begin{array}{ll}
		\ddot{x}(t) + \alpha\dot{x}(t) + \nabla f \left( x(t) + \beta\dot{x}(t) \right) = 0,
	\end{array}
\end{align}
where  $\alpha > 0$, $\beta \geq 0$ and $f: \mathbb{R}^{n} \rightarrow \mathbb{R}$ is a $\gamma$-strongly quasiconvex function with modulus $\gamma > 0$. We also consider the  system \eqref{system1}   with external perturbations:
\begin{align}\label{system2}
	\begin{array}{ll}
		\ddot{x}(t) + \alpha\dot{x}(t) + \nabla f \left(x(t) + \beta\dot{x}(t) \right) = \epsilon(t),
	\end{array}
\end{align}
where $\epsilon: [t_0, + \infty) \rightarrow \mathbb{R}^{n}$ serves as an external perturbation which satisfies $\int_{t_{0}}^{+\infty} \|\epsilon(t)\|^2 dt < +\infty$. Specifically, we demonstrate that the systems \eqref{system1} and \eqref{system2} can be explicitly discretized to obtain the corresponding inertial accelerated algorithms. The contributions of this paper can be more specifically stated as follows:
\begin{enumerate}
	\item[{\rm (i)}] \textup{Under some mild conditions on the parameters, we establish exponential convergence rates of   the objective
function value error, the trajectory and the velocity vector for the dynamical system with implicit Hessian-driven damping \eqref{system1}, without requiring strongly convexity and Lipschitz continuous   gradient assumptions on the objective function.}
	\item[{\rm (ii)}]\textup{By temporal discretization of the systems \eqref{system1} and \eqref{system2}, we obtain inertial accelerated algorithm \eqref{algo} and its externally perturbed version \eqref{algo-error}, respectively. We also establish convergence rates of both algorithms under suitable conditions on the parameters.}
	\item[{\rm (iii)}]\textup{Through numerical experiments, our algorithm \eqref{algo}  is shown to exhibit faster convergence and less oscillations than the existing algorithms: \eqref{heavyball}, \eqref{NAG}, \eqref{heavyballhess}, and \eqref{NAGhess}.}
\end{enumerate}

The rest of the paper is organized as follows. In Section 2, we recall some basic notions and present some preliminary results. In Section 3, we investigate the asymptotic properties of the inertial dynamical systems \eqref{system1} and   \eqref{system2}. In Section 4, we analyze the convergence rates of the algorithms obtained by the temporal discretization of  \eqref{system1} and \eqref{system2}. In Section 5, we give some numerical experiments to illustrate the effectiveness of our algorithms.

\section{Preliminaries}

Throughout this paper, let $\mathbb{R}^n$ be an $n$-dimensional Euclidean space with inner product $\langle \cdot,\cdot \rangle$ and norm $\lVert \cdot \rVert$. Let $\psi: \mathbb{R}^{n} \rightarrow \mathbb{R}$  be a continuous differentiable function such that $\nabla \psi$ is $L$-Lipschitz continuous for $L>0.$ Then, it holds  (see \cite[Theorem 2.1.5]{convexbook}):
\begin{equation}  \label{L}
	\psi(y) \leq \psi(x) + \langle \nabla \psi(x), y - x \rangle + \frac{L}{2} \|x - y \|^{2},~\forall x, y \in \mathbb{R}^{n}.
\end{equation}
For an extended real-valued function $f: \mathbb{R}^n \rightarrow \mathbb{R}\cup\{+\infty\}$, we define the domain and epigraph of $f$ by
\begin{eqnarray*}
	{\rm dom}\,f:=\left\{x\in\mathbb{R}^n\mid f(x) <+\infty\right\}~\text{and}~{\rm epi}\,f:=\left\{(x, \mu)\in\mathbb{R}^n\times\mathbb{R}~|~\mu\geq f(x)\right\},
\end{eqnarray*}
respectively. The function $f$ is said to be proper iff ${\rm dom}\,f$ is nonempty, and $f$ is said to be convex iff ${\rm epi}\,f$ is a convex set.
Moreover, $f$ is said to be lower semicontinuous iff ${\rm epi}\,f$ is closed. We say that $f$ with convex domain is $\gamma$-strongly quasiconvex with modulus $\gamma > 0$ iff for any $x, y \in {\rm dom}\,f$ and $\lambda \in [0,1]$,
\begin{equation*}\label{def:s qcx}
	f(\lambda y + (1-\lambda)x) \leq \max \{f(y), f(x)\} - \lambda(1 - \lambda) \frac{\gamma}{2} \| x - y \|^{2}.
\end{equation*}
Obviously, every strongly convex function is strongly quasiconvex function. The reverse statement does not hold in general, see \cite{cambini2009generalized, hadjisavvas2006handbook, lara2022strongly}. Further, if $K \subseteq {\rm dom}\,f$ is a closed convex set and $f: K \rightarrow {\mathbb{R}}\cup\{+\infty\}$ is a proper, lower semicontinuous and $\gamma$-strongly quasiconvex function with modulus $\gamma> 0$. Then, ${\rm argmin}_{K} f$ is a singleton, please see \cite[Corollary 3]{lara2022strongly}. For more details on the relationship between strong quasiconvexity and other kinds of convexity, we refer the readers to \cite{cambini2009generalized, sofar, hadjisavvas2006handbook, lara2024characterizations}.

Now, we recall the following important properties of the  strongly quasiconvexity which will be used in the sequel.
\begin{lemma} \label{4} (\cite[Theorem 5.1]{nam2024strong})
	Suppose that $f$ is lower semicontinuous and $\gamma$-strongly quasiconvex on a closed convex set $K \subseteq \mathbb{R}^{n}$. Then, for $x^{*} = {\rm argmin}_{K}\,f$, we have
	\begin{equation*}
		f(x^{*}) + \frac{\gamma}{4} \|x - x^{*}\|^2 \leq f(x), ~\forall x \in K.
	\end{equation*}
\end{lemma}

\begin{lemma} (\cite[Theorems 2 and 6]{vladimirov1978uniformly})
	Let $K \subseteq \mathbb{R}^{n}$ be a convex set and $f: K \rightarrow {\mathbb{R}}$ be a differentiable function. Then, $f$ is $\gamma$-strongly quasiconvex function with modulus $\gamma>0$ if and only if for any $x, y \in K$,
	\begin{equation} \label{x<y}
		f(x) \leq f(y) ~ \Longrightarrow ~ \langle \nabla f(y), x - y \rangle
		\leq -\frac{\gamma}{2} \| y - x \|^{2}.
	\end{equation}
\end{lemma}

 Let $f: \mathbb{R}^n \to \mathbb{R}$ be a differentiable function and $x^{*} \in \arg\min_{x \in \mathbb{R}^n} f$. Following  \cite{Lojasiewicz1963,PLpolyak}, we say that  $f$   satisfies the Polyak-{\L}ojasiewicz property iff  there exists a constant $\eta > 0$ such that for any $ x \in \mathbb{R}^n$,  $\|\nabla f(x)\|^2 \ge 2\eta \left( f(x) - f(x^{*}) \right).$ 

\begin{lemma} \label{PL} (\cite[Theorem 2]{korablev1989relaxation})
	Let $f: \mathbb{R}^{n} \rightarrow \mathbb{R}$ be a $\gamma$-strongly quasiconvex function with modulus $\gamma > 0$ and differentiable with $L$-Lipschitz continuous gradient with modulus $L>0$.   Then, for $x^{*} = {\rm argmin}_{\mathbb{R}^n}\, f$,  the Polyak-{\L}ojasiewicz property holds with modulus $\eta := \frac{\gamma^{2}}{4L}$,  that is,
	\begin{equation*}
		\| \nabla f(x) \|^{2} \geq \frac{\gamma^{2}}{2L} (f(x) - f(x^{*})),~\forall x \in \mathbb{R}^n.
	\end{equation*}
\end{lemma}

The following   important properties will be used in the sequel.
\begin{lemma} \label{7} (\cite[Lemma A3]{tikhonov})
	Suppose that $\delta > 0$, $\phi \in L^1(\delta, +\infty)$ is a nonnegative and continuous function, and $\varphi: [t_0, +\infty) \to [0, +\infty)$ is a  nondecreasing function such that $\lim\limits_{t \to +\infty} \varphi(t) = +\infty $. Then, $$\lim\limits_{t \to +\infty} \frac{1}{\varphi(t)} \int_{t_0}^t \varphi(s) \phi(s) ds = 0. $$
\end{lemma}

\begin{lemma} \label{jishu}
	Let $0 < \theta < 1$ and $q > 0$. Let $\{a_i\} \subset \mathbb{R}$ be a sequence satisfying $a_i = \mathcal{O} (\frac{1}{i^q})$, as $i \rightarrow +\infty$. Then, $S_k := \sum_{i=1}^{k}\theta^{k-i} a_i = \mathcal{O} (\frac{1}{k^q})$, as $k \rightarrow +\infty$.
\end{lemma}

\begin{proof}
	Let us calculate $S_k = I_1 + I_2$, where
	\begin{align*}
		I_1 := \sum_{i=1}^{\lfloor \frac{k}{2} \rfloor} \theta^{k-i} a_i ~~ \text{and}~~ I_2 := \sum_{i=\lfloor \frac{k}{2} \rfloor + 1}^{k} \theta^{k-i} a_i.
	\end{align*}
	Here $\lfloor \frac{k}{2} \rfloor$ denotes the greatest integer less than or equal to $\frac{k}{2}$.

Now, we consider the following two cases:
	
	(i) For $1 \leq i \leq \lfloor \frac{k}{2} \rfloor$, we have $k - i \geq k - \lfloor \frac{k}{2} \rfloor \geq \frac{k}{2}$. Since $0 < \theta < 1$, it is obvious that $\theta^{k-i} \leq \theta^{\frac{k}{2}}$. By $a_i = \mathcal{O} (\frac{1}{i^q})$, as $i \rightarrow +\infty$, there exists $C > 0 $ such that for any $i \geq 1$, $a_i \leq \frac{C}{i^q} \leq C$. Then,
    \begin{align*}
    	I_1 \leq \sum_{i=1}^{\lfloor \frac{k}{2} \rfloor} \theta^{k-i}   C \leq C\sum_{i=1}^{\lfloor \frac{k}{2} \rfloor} \theta^{\frac{k}{2}} = C \theta^{\frac{k}{2}}  \lfloor k/2 \rfloor \leq C k \theta^{\frac{k}{2}}.
    \end{align*}

	\noindent It is easy to know that for any $q > 0$, $k  \theta^{\frac{k}{2}} = o\left(\frac{1}{k^q} \right),$  as $ k \rightarrow +\infty$. Therefore, $ I_1 = o\left( \frac{1}{k^q} \right), ~\textup{as} ~ k \rightarrow +\infty$.
	
	(ii) For $ \lfloor \frac{k}{2} \rfloor + 1 \leq i \leq k$, let $j = k - i $. Then,
$
		I_2 = \sum_{j=0}^{k - \lfloor \frac{k}{2} \rfloor - 1} \theta^j  a_{k-j}.
$
 Since $k - j=i \geq \lfloor \frac{k}{2} \rfloor + 1 \geq \frac{k}{2}$, we have
	\[
	a_{k-j} \leq \frac{C}{(k - j)^q} \leq \frac{C}{(\frac{k}{2})^q} = \frac{2^q C}{k^q}.
	\]
	Therefore,
	\[
	I_2 \leq \sum_{j=0}^{k - \lfloor \frac{k}{2} \rfloor - 1} \theta^j  \frac{2^q C}{k^q} \leq \frac{2^q C}{1 - \theta}   \frac{1}{k^q},
	\]
	which implies that $I_2 =\mathcal{O} \left( \frac{1}{k^q} \right),~\textup{as} ~ k \rightarrow +\infty$.
	
	Now, combining (i) and (ii), we conclude that $	S_k = \mathcal{O}\left( \frac{1}{k^q} \right),~\textup{as} ~ k \rightarrow +\infty$.
	The proof is complete. \qed
\end{proof}

\section{The continuous-time dynamics approach}

In this section, we investigate the asymptotic properties of the system \eqref{system1} and its perturbed version \eqref{system2}.In the sequel, we need the  following weak assumption which has been used in \cite{hadjisavvas2025heavy, lara2024characterizations}.
 \begin{assumption} \label{assume}
	For $x^{*} = {\rm argmin}_{\mathbb{R}^n} f$, there exists a constant $\kappa > 0$ such that for any   $x\in \mathbb{R}^n$,
	\begin{equation*}
		\langle \nabla f(x ), x  - x^{*} \rangle \geq \kappa (f(x ) - f(x^{*})).
	\end{equation*}
\end{assumption} 

\begin{remark}
\begin{enumerate}
\item[{\rm (i)}] In the case that $f$ is a convex or strongly convex function, Assumption \ref{assume} holds trivially with $\kappa=1$. Thus, Assumption \ref{assume} can be viewed as a generalized convexity assumption. It is worth noting that for differentiable strongly quasiconvex functions, sufficient conditions for Assumption \ref{assume} to hold can be found in \cite[Proposition 10 and Corollary 11]{hadjisavvas2025heavy}.
\item[{\rm (ii)}] For the case of  $0<\kappa\leq1$,  Assumption \ref{assume} reduces to the concept of $\kappa$-quasar convexity introduced in \cite[Definition 1]{hin}, which has been studied deeply because of its applications in machine learning theory.
\item[{\rm (iii)}] In the case that $f$ is a  $\gamma$-strongly quasiconvex function with  $L$-Lipschitz continuous  gradient, it is easy to show that
Assumption \ref{assume} holds trivially with $\kappa=\frac{\gamma}{L}$. We refer the readers to \cite[Page 16]{lara2024characterizations} for more details.
\end{enumerate}
\end{remark}

By virtue of Assumption \ref{assume}, we establish exponential convergence rates for the objective function value error, the trajectory and the velocity vector along the trajectory generated by the system \eqref{system1}.
\begin{theorem} \label{lianxu}
	Let $f: \mathbb{R}^n \rightarrow \mathbb{R}$ be a $\gamma$-strongly quasiconvex function with modulus $\gamma > 0$ and let $x: [t_0, + \infty) \rightarrow \mathbb{R}^{n}$ be a solution of the system \eqref{system1}. Suppose that Assumption \textup{\ref{assume}} holds and $x^{*} = {\rm argmin}_{\mathbb{R}^n} f$. Then, for any
	$\alpha  \in \, \left( 0, \frac{\kappa + 4}{4}\sqrt{\frac{\gamma}{\kappa}}\right]$ and
	$\beta \in \, \left[ 0, \frac{\sqrt{\alpha^2(\kappa+2)^4+16\gamma(\kappa+4)^3}-\alpha(\kappa+2)^2}{4\gamma(\kappa+4)} \right]$, it holds that
	\begin{align*}
		\left\{
		\begin{array}{lll}
			f(x(t)+\beta\dot{x}(t)) - f(x^{*}) = \mathcal{O} (e^{-\frac{\alpha \kappa}{\kappa+4}t}), ~\textup{as}~ t \rightarrow +\infty,\\ [2mm]
			\|x(t) - x^{*}\| = \mathcal{O} (e^{-\frac{\alpha \kappa}{2(\kappa+4)}t}), ~\textup{as}~ t \rightarrow +\infty,\\ [2mm]
			\|\dot{x}(t)\|=\mathcal{O}(e^{-\frac{\alpha \kappa}{2(\kappa+4)}t}), ~\textup{as}~ t \rightarrow +\infty.
		\end{array}
		\right.
	\end{align*}
\end{theorem}

\begin{proof}
	 We introduce the energy function $\mathcal{E}: [0, + \infty) \rightarrow \mathbb{R}$ as follows:
	\begin{align} \label{nlfunc}
		\mathcal{E}(t) := f(x(t)+\beta\dot{x}(t)) - f(x^{*}) + \frac{1}{2} \| v(t) \|^2 +\frac{\lambda^2}{2}\|x(t) - x^{*} \|^2,
	\end{align}
	where $v(t) := \lambda (x(t) - x^{*}) + \dot{x}(t)$ and $\lambda:= \frac{2\alpha}{\kappa + 4}$.
	
	We now analyze the time derivative of $\mathcal{E}(t)$. Clearly,
	\begin{align*}
		\dot{v}(t)= \lambda \dot{x}(t)+\ddot{x}(t) = (\lambda - \alpha) \dot{x}(t) - \nabla f(x(t)+\beta\dot{x}(t)).
	\end{align*}
	Thus,
	\begin{equation} \label{v}
		\begin{split}
			\langle v(t), \dot{v}(t) \rangle = & ~ \lambda(\lambda-\alpha)\langle x(t)-x^{*}, \dot{x}(t) \rangle- \lambda \langle \nabla f(x(t)+\beta\dot{x}(t)), x(t)-x^{*} \rangle \\
			& + (\lambda-\alpha) \|\dot{x}(t)\|^2  - \langle \nabla f(x(t)+\beta\dot{x}(t)), \dot{x}(t) \rangle.
		\end{split}
	\end{equation}
	Note that
	\begin{align} \label{ky}
		\begin{split}
			&\langle \nabla f(x(t)+\beta\dot{x}(t)), x (t) - x^{*} \rangle\\ = & ~ \langle \nabla f(x(t)+\beta\dot{x}(t)),x(t)+\beta\dot{x}(t)-x^{*} \rangle - \langle \nabla f(x(t)+\beta\dot{x}(t)),\beta\dot{x}(t) \rangle \\
			\geq & ~  \frac{\gamma}{4} \|x(t)+\beta\dot{x}(t) - x^{*}\|^2 + \frac{\kappa}{2} \big(f(x(t)+\beta\dot{x}(t))-f(x^{*})\big) \\
			& - \langle \nabla f(x(t)+\beta\dot{x}(t)),\beta\dot{x}(t) \rangle,
		\end{split}
	\end{align}
	where the inequality holds due to \eqref{x<y} and Assumption \ref{assume}. Together with \eqref{v} and \eqref{ky}, we get
	\begin{equation} \label{1}
		\begin{split}
			\langle v(t), \dot{v}(t) \rangle
			\leq & ~ \lambda(\lambda-\alpha)\langle x(t)-x^{*}, \dot{x}(t) \rangle - \frac{\lambda\gamma}{4}\|x(t)+\beta\dot{x}(t)-x^{*}\|^2  \\
			&  -\frac{\lambda\kappa}{2}\big(f(x(t)+\beta\dot{x}(t))-f(x^{*})\big) + (\lambda-\alpha)\|\dot{x}(t)\|^2  \\
			& + (\lambda\beta-1) \langle \nabla f(x(t)+\beta\dot{x}(t)), \dot{x}(t) \rangle.
		\end{split}
	\end{equation}

\noindent On the other hand, it is easy to show that
\begin{align*}
	\dot{\mathcal{E}} (t) = & ~ \langle  \nabla f(x(t)+\beta\dot{x}(t)), \dot{x}(t)+\beta\ddot{x}(t) \rangle + \langle v(t), \dot{v}(t) \rangle + \lambda^2 \langle x(t)-x^{*}, \dot{x}(t) \rangle \notag \\
	= & ~ \langle \nabla f(x(t)+\beta\dot{x}(t)), (1-\alpha\beta)\dot{x}(t)-\beta \nabla f(x(t)+\beta\dot{x}(t)) \rangle + \langle v(t), \dot{v}(t) \rangle \notag \\
	& + \lambda^2 \langle x(t)-x^{*}, \dot{x}(t) \rangle.
\end{align*}
Then, together with \eqref{1}, we have
\begin{align*}
	\dot{\mathcal{E}} (t) \leq & ~ \beta(\lambda-\alpha) \langle  \nabla f(x(t)+\beta\dot{x}(t)), \dot{x} (t) \rangle -\beta \|\nabla f(x(t)+\beta\dot{x}(t)) \|^2 \notag \\
	& + \lambda(2\lambda-\alpha) \langle x(t)-x^{*}, \dot{x}(t) \rangle - \frac{\lambda\gamma}{4}\|x(t)+\beta\dot{x}(t)-x^{*}\|^2 \notag \\
	& - \frac{\lambda\kappa}{2}\big(f(x(t)+\beta\dot{x}(t))-f(x^{*})\big) + (\lambda-\alpha) \|\dot{x}(t)\|^2.
\end{align*}

\noindent Taking \eqref{nlfunc} into account, we have
\begin{equation} \label{e(t)}
	\begin{split}
		\dot{\mathcal{E}} (t) \leq & ~ \beta(\lambda-\alpha) \langle  \nabla f(x(t)+\beta\dot{x}(t)), \dot{x} (t) \rangle -\beta \|\nabla f(x(t)+\beta\dot{x}(t)) \|^2 \\
		& + \lambda(2\lambda-\alpha) \langle 	x(t)-x^{*}, \dot{x}(t) \rangle - \frac{\lambda\gamma}{4}\|x(t)+\beta\dot{x}(t)-x^{*}\|^2  \\
		& - \frac{\lambda\kappa}{2} \left( 	\mathcal{E}(t)- \frac{1}{2}\| v(t) \|^2 - \frac{\lambda^2}{2}\|x(t) - x^{*} \|^2 \right) + (\lambda-\alpha) \|\dot{x}(t) \|^2  \\
		= & ~ \beta(\lambda-\alpha) \langle  \nabla 	f(x(t)+\beta\dot{x}(t)), \dot{x} (t) \rangle -\beta \|\nabla f(x(t)+\beta\dot{x}(t)) \|^2  \\
		& + \lambda \left( 2\lambda-\alpha + \frac{\lambda\kappa}{2} \right) \langle 	x(t)-x^{*}, \dot{x}(t) \rangle - \frac{\lambda\gamma}{4}\|x(t)+\beta\dot{x}(t)-x^{*}\|^2  \\
		& - \frac{\lambda\kappa}{2}\mathcal{E}(t) + \frac{\lambda^3\kappa}{2} \|x(t)-x^{*} \|^2 + \left(\lambda - \alpha + \frac{\lambda\kappa}{4}\right) \|\dot{x}(t)\|^2.
	\end{split}
\end{equation}
 Note that
\begin{align*}
	\beta(\lambda-\alpha)\langle \nabla f(x(t)+\beta\dot{x}(t)), \dot{x} (t) \rangle \leq \beta \|\nabla f(x(t)+\beta\dot{x}(t))\|^2 + \frac{\beta(\lambda-\alpha)^2}{4}\|\dot{x}(t)\|^2
\end{align*}
and
\begin{align*}
	-\frac{\lambda\gamma}{4}\|x(t)+\beta\dot{x}(t)-x^{*}\|^2 &\leq -\frac{\lambda\gamma}{4} \left( \frac{1}{2} \|x(t)-x^{*}\|^2 - \beta^2 \|\dot{x}(t)\|^2 \right).
\end{align*}

\noindent Then, it follows from \eqref{e(t)} that
\begin{equation} \label{2}
	\begin{split}
		&\dot{\mathcal{E}} (t) + \frac{\lambda \kappa}{2} \mathcal{E} (t)\\
		\leq & ~ \lambda \left( 2\lambda-\alpha+\frac{\lambda \kappa}{2} \right) \langle x(t)-x^{*}, \dot{x}(t) \rangle + \frac{\lambda}{4}\left(2\lambda^2\kappa-\frac{\gamma}{2}\right) \|x(t)-x^{*} \|^2 \\
		& + \left(   \frac{\lambda\gamma\beta^2}{4} + \frac{\beta(\lambda-\alpha)^2}{4} + \lambda - \alpha + \frac{\lambda\kappa}{4} \right) \|\dot{x}(t) \|^2.
	\end{split}
\end{equation}

Next, we analyze the coefficients of the right side of \eqref{2}. Clearly, from $\lambda = \frac{2 \alpha}{\kappa + 4}$, we have $2\lambda-\alpha+\frac{\lambda \kappa}{2}=0$. Since $0 < \alpha \leq \frac{\kappa + 4}{4}\sqrt{\frac{\gamma}{\kappa}}$, we have
\begin{align*}
	2\lambda^2\kappa-\frac{\gamma}{2} = \frac{8\alpha^2\kappa}{(\kappa+4)^2} - \frac{\gamma}{2} \leq 0.
\end{align*}
Further,   from $0 \leq \beta \leq \frac{\sqrt{\alpha^2(\kappa+2)^4+16\gamma(\kappa+4)^3}-\alpha(\kappa+2)^2}{4\gamma(\kappa+4)}$, we can easily get
\begin{align*}
	 \frac{\lambda\gamma\beta^2}{4} + \frac{\beta(\alpha-\lambda)^2}{4} + \lambda - \alpha + \frac{\lambda\kappa}{4} = \frac{\alpha\gamma}{2(\kappa+4)}\beta^2 + \frac{\alpha^2(\kappa+2)^2}{4(\kappa+4)^2}\beta -\frac{\alpha}{2} \leq 0.
\end{align*}
Therefore, it follows from \eqref{2} that
\begin{equation*}
	\dot{\mathcal{E}} (t) + \frac{\lambda \kappa}{2} \mathcal{E} (t) \leq 0.
\end{equation*}

\noindent Multiplying $e^{\frac{\lambda\kappa}{2}t}$ on both side of the above inequality and integrating it from $t_0$ to $t$ where $t \geq t_0$, we obtain
\begin{equation*}
	\mathcal{E}(t) \leq \mathcal{E} (t_0) \, e^{- \frac{\lambda\kappa}{2} (t-t_0)}.
\end{equation*}

\noindent This together with \eqref{nlfunc} and $\lambda= \frac{2\alpha}{\kappa + 4}$ yields
\begin{align}
	& f(x(t)+\beta\dot{x}(t)) - f(x^{*}) = \mathcal{O} (e^{-\frac{\alpha \kappa}{\kappa+4}t}), ~\textup{as}~ t \rightarrow +\infty, \notag \\
	& \|x(t) - x^{*}\| = \mathcal{O} (e^{-\frac{\alpha \kappa}{2(\kappa+4)}t}), ~\textup{as}~ t \rightarrow +\infty, \label{x-x*}
\end{align}
and
\begin{align*}
	\| \lambda (x(t) - x^{*}) + \dot{x}(t) \|=\mathcal{O}(e^{-\frac{\alpha \kappa}{2(\kappa+4)}t}),~\textup{as}~ t \rightarrow +\infty.
\end{align*}

\noindent By \eqref{x-x*} and $\|\dot{x}(t)\| \leq \|\lambda (x(t) - x^{*}) + \dot{x}(t)\| + \|\lambda(x(t) - x^{*})\|$, we obtain
\begin{align*}
	\|\dot{x}(t)\|=\mathcal{O}(e^{-\frac{\alpha \kappa}{2(\kappa+4)}t}), ~\textup{as}~ t \rightarrow +\infty.
\end{align*}
The proof is complete.\qed
\end{proof}

\begin{remark}
\begin{enumerate}
\item[{\rm (i)}] For  the case where $\alpha=\beta=0$ in  system \eqref{system1}, an exponential convergence rate of the trajectory was established in \cite[Theorem 3.4]{quasiconvex2021} under the condition that $f$ is a strongly quasiconvex function and $\nabla f$ is   Lipschitz continuous. However, Theorem \ref{lianxu} establishes exponential convergence rates of objective function value error, the trajectory and the velocity vector for system \eqref{system1} without relying on the Lipschitz continuity assumption of $\nabla f$. Hence, Theorem \ref{lianxu} can be regarded as a generalization of \cite[Theorem 3.4]{quasiconvex2021}.

\item[{\rm (ii)}]  Although \cite[Theorem 5]{hadjisavvas2025heavy} obtained some exponential convergence results for   inertial dynamical systems with \emph{explicit} Hessian-driven damping   for differentiable strongly quasiconvex functions, Theorem \ref{lianxu} is the first result to establish exponential convergence rates for   systems  incorporating \emph{implicit} Hessian-driven damping, applied to the same class of functions. However, it is worth noting that  the   convergence rate obtained in Theorem \ref{lianxu}  is not as fast as that in \cite[Theorem 5]{hadjisavvas2025heavy}.  The reason appears to be the inclusion of the term $\frac{\lambda^2}{2}\|x(t) - x^{*} \|^2$ in the energy function, which is essential for ensuring the convergence rate of trajectories  to    the minimizer of problem \eqref{min.f}
    in systems with \emph{implicit} Hessian-driven damping.
\end{enumerate}
\end{remark}

  In the sequel, based on Assumption \textup{\ref{assume}}, we analyze the asymptotic properties of the system \eqref{system2}.

\begin{theorem} \label{3.2}
	Let $f: \mathbb{R}^n \rightarrow \mathbb{R}$ be a $\gamma$-strongly quasiconvex function with modulus $\gamma > 0$ and let $x: [t_0, + \infty) \rightarrow \mathbb{R}^{n}$ be a solution of the system \eqref{system2}. Suppose that Assumption \textup{\ref{assume}} holds, $\int_{t_{0}}^{+\infty} \|\epsilon(t)\|^2 dt < +\infty$ and $x^{*} = {\rm argmin}_{\mathbb{R}^n} f$. Then, for any
	$\alpha \in \, \left( 0, \frac{\kappa + 4}{4}\sqrt{\frac{\gamma}{2\kappa}}\right]$ and   $\beta \in \, \left[ 0, \frac{\sqrt{2\alpha^2(\kappa+2)^4+27\gamma(\kappa+4)^3}-\sqrt{2}\alpha(\kappa+2)^2}{9\sqrt{2}\gamma(\kappa+4)} \right]$, the following properties are satisfied:
	\begin{enumerate}
	\item[{\rm (i)}] \textup{(Minimizing properties):
	\begin{align*}
	\lim\limits_{t \to +\infty} f(x(t)+\beta\dot{x}(t)) = f(x^{*})~ \text{and} \lim\limits_{t \to +\infty} \|x(t) - x^{*}\| = \lim\limits_{t \to +\infty} \|\dot{x}(t)\| = 0.
	\end{align*}}
	\item[{\rm (ii)}] \textup{(Convergence rates): Assume further that for some $p>0$, $\|\epsilon(t)\|= \mathcal{O} \left( \frac{1}{t^{p}} \right)$, as $t \rightarrow +\infty$. Then,
	\begin{align*}
		\left\{
		\begin{array}{lll}
			f(x(t)+\beta\dot{x}(t)) - f(x^{*}) = \mathcal{O} \left( \frac{1}{t^{2p}} \right), ~\textup{as}~ t \rightarrow +\infty,\\ [2mm]
			\|x(t) - x^{*}\| = \mathcal{O} \left( \frac{1}{t^{p}} \right), ~\textup{as}~ t \rightarrow +\infty,\\ [2mm]
			\|\dot{x}(t)\|=\mathcal{O} \left( \frac{1}{t^{p}} \right), ~\textup{as}~ t \rightarrow +\infty.
		\end{array}
		\right.
	\end{align*}}
	\end{enumerate}
\end{theorem}

\begin{proof}
	We consider the energy function $\mathcal{E}(t)$ defined in \eqref{nlfunc}. Using a similar argument to  \eqref{e(t)},  we can easily get
	\begin{equation} \label{e(t)2}
		\begin{split}
			\dot{\mathcal{E}} (t) \leq & ~ \beta(\lambda-\alpha) \langle  \nabla f(x(t)+\beta\dot{x}(t)), \dot{x} (t) \rangle -\beta \|\nabla f(x(t)+\beta\dot{x}(t)) \|^2 \\
			& + \lambda(2\lambda-\alpha) \langle 	x(t)-x^{*}, \dot{x}(t) \rangle - \frac{\lambda\gamma}{4}\|x(t)+\beta\dot{x}(t)-x^{*}\|^2  \\
			& - \frac{\lambda\kappa}{2} \left( 	\mathcal{E}(t)- \frac{1}{2}\| v(t) \|^2 - \frac{\lambda^2}{2}\|x(t) - x^{*} \|^2 \right) + (\lambda-\alpha) \|\dot{x}(t) \|^2  \\
			& + \beta \langle \nabla f(x(t)+\beta\dot{x}(t)), \epsilon(t) \rangle  + \lambda \langle x(t)-x^{*}, \epsilon(t) \rangle + \langle \dot{x}(t), \epsilon(t) \rangle \\
			= & ~ \beta(\lambda-\alpha) \langle  \nabla 	f(x(t)+\beta\dot{x}(t)), \dot{x} (t) \rangle -\beta \|\nabla f(x(t)+\beta\dot{x}(t)) \|^2  \\
			& + \lambda \left( 2\lambda-\alpha + \frac{\lambda\kappa}{2} \right) \langle 	x(t)-x^{*}, \dot{x}(t) \rangle - \frac{\lambda\gamma}{4}\|x(t)+\beta\dot{x}(t)-x^{*}\|^2  \\
			& - \frac{\lambda\kappa}{2}\mathcal{E}(t) + \frac{\lambda^3\kappa}{2} \|x(t)-x^{*} \|^2 + \left(\lambda - \alpha + \frac{\lambda\kappa}{4}\right) \|\dot{x}(t)\|^2 \\
			& + \beta \langle \nabla f(x(t)+\beta\dot{x}(t)), \epsilon(t) \rangle  + \lambda \langle x(t)-x^{*}, \epsilon(t) \rangle + \langle \dot{x}(t), \epsilon(t) \rangle.
		\end{split}
	\end{equation}
	Note that
\begin{equation*}
	\left\{\begin{split}
		\beta(\lambda-\alpha)\langle \nabla f(x(t)+\beta\dot{x}(t)), \dot{x} (t) \rangle & \leq \frac{3\beta}{4} \|\nabla f(x(t)+\beta\dot{x}(t))\|^2 + \frac{\beta(\lambda-\alpha)^2}{3}\|\dot{x}(t)\|^2,\\
		-\frac{\lambda\gamma}{4}\|x(t)+\beta\dot{x}(t)-x^{*}\|^2 & \leq -\frac{3\lambda\gamma}{4} \left( \frac{1}{4} \|x(t)-x^{*}\|^2 - \beta^2 \|\dot{x}(t)\|^2 \right),\\
 \beta \langle \nabla f(x(t)+\beta\dot{x}(t)), \epsilon(t) \rangle&\leq \frac{\beta}{4}\|\nabla 	f(x(t)+\beta\dot{x}(t)) \|^2+\beta\|\epsilon(t)\|^2,\\
		   \lambda \langle x(t)-x^{*}, \epsilon(t) \rangle&\leq \frac{\lambda\gamma}{8} \|x(t)-x^{*}\|^2 +\frac{2\lambda}{\gamma}\|\epsilon(t)\|^2,\\
  \langle \dot{x}(t), \epsilon(t) \rangle		& \leq  \frac{\alpha}{4} \|\dot{x}(t)\|^2 +  \frac{1}{\alpha} \|\epsilon(t)\|^2.
	\end{split}\right.
\end{equation*}
	Then, it follows from \eqref{e(t)2} that
	\begin{equation} \label{3}
		\begin{split}
			\dot{\mathcal{E}} (t) + \frac{\lambda \kappa}{2} \mathcal{E} (t)
			\leq & ~ \lambda \left( 2\lambda-\alpha+\frac{\lambda \kappa}{2} \right) \langle x(t)-x^{*}, \dot{x}(t) \rangle \\
			& + \frac{\lambda}{4} \left(2\lambda^2\kappa-\frac{\gamma}{4} \right) \|x(t)-x^{*} \|^2 + \left(\beta+\frac{1}{\alpha}+\frac{2\lambda}{\gamma} \right)\|\epsilon(t)\|^2 \\
			& + \left(  \frac{3\lambda\gamma\beta^2}{4} + \frac{\beta(\lambda-\alpha)^2}{3} +  \lambda - \frac{3\alpha}{4} + \frac{\lambda\kappa}{4} \right) \|\dot{x}(t)\|^2.
		\end{split}
	\end{equation}
	
	Next, we analyze the coefficients of the right side of \eqref{3}. Clearly, from $\lambda = \frac{2 \alpha}{\kappa + 4}$, we have $2\lambda-\alpha+\frac{\lambda \kappa}{2}=0$. Moreover, by $0 < \alpha \leq \frac{\kappa + 4}{4}\sqrt{\frac{\gamma}{2\kappa}}$, we have
	\begin{align*}
		2\lambda^2\kappa-\frac{\gamma}{4} = \frac{8\alpha^2\kappa}{(\kappa+4)^2} - \frac{\gamma}{4} \leq 0.
	\end{align*}
	Further,  from $0 \leq \beta \leq \frac{\sqrt{2\alpha^2(\kappa+2)^4+27\gamma(\kappa+4)^3}-\sqrt{2}\alpha(\kappa+2)^2}{9\sqrt{2}\gamma(\kappa+4)}$ and $\lambda = \frac{2 \alpha}{\kappa + 4}$, we have
	\begin{align*}
	    \frac{3\lambda\gamma\beta^2}{4} + \frac{\beta(\lambda-\alpha)^2}{3} +  \lambda - \frac{3\alpha}{4} + \frac{\lambda\kappa}{4} = \frac{3\alpha\gamma}{2(\kappa+4)}\beta^2 + \frac{\alpha^2(\kappa+2)^2}{3(\kappa+4)^2}\beta - \frac{\alpha}{4} \leq 0.
    \end{align*}
    Thus, it follows from \eqref{3} that
	\begin{align*}
		\dot{\mathcal{E}} (t) + \frac{\lambda \kappa}{2} \mathcal{E} (t)  \leq C\|\epsilon(t)\|^2,
	\end{align*}
	where $C:= \beta + \frac{1}{\alpha} + \frac{2\lambda}{\gamma}$. Multiplying $e^{\frac{\lambda\kappa}{2}t}$ on both side of the above inequality and integrating it from $t_0$ to $t$, where $t \geq t_0$, we obtain
	\begin{align} \label{integrate}
		{\mathcal{E}} (t) \leq {\mathcal{E}} (t_{0}) e^{-\frac{\lambda \kappa}{2}(t-t_{0})} + Ce^{-\frac{\lambda \kappa}{2}t}\int_{t_{0}}^{t} e^{\frac{\lambda \kappa}{2}\tau}\| \epsilon(\tau)\|^2 d\tau.
	\end{align}
	
	(i) Since $\int_{t_{0}}^{+\infty} \|\epsilon(t)\|^2 dt < +\infty$, we deduce from Lemma \ref{7} that $	\lim\limits_{t \to +\infty} {\mathcal{E}} (t) = 0$. This together with \eqref{nlfunc} yields
	\[
	\lim\limits_{t \to +\infty} f(x(t)+\beta\dot{x}(t)) = f(x^{*}) ~\text{and} \lim\limits_{t \to +\infty}  \|x(t)-x^{*}\| = \lim\limits_{t \to +\infty}  \|\dot{x}(t)\| = 0.
	\]
	
	(ii) Clearly, using a similar argument to that in \cite[Theorem 4.1 (ii)]{attouch2021effect}, we can easily deduce from \eqref{integrate} that
	\begin{align*}
		{\mathcal{E}} (t) = {\mathcal{O}}\left( \frac{1}{t^{2p}} \right), ~\textup{as}~ t \rightarrow +\infty.
	\end{align*}
	Then, using a similar argument to the one presented in Theorem \ref{lianxu}, we have
	\begin{align*}
		\left\{
		\begin{array}{lll}
			f(x(t)+\beta\dot{x}(t)) - f(x^{*}) = \mathcal{O} \left( \frac{1}{t^{2p}} \right), ~\textup{as}~ t \rightarrow +\infty,\\ [2mm]
			\|x(t) - x^{*}\| = \mathcal{O} \left( \frac{1}{t^{p}} \right), ~\textup{as}~ t \rightarrow +\infty,\\ [2mm]
			\|\dot{x}(t)\|=\mathcal{O} \left( \frac{1}{t^{p}} \right), ~\textup{as}~ t \rightarrow +\infty.
		\end{array}
		\right.
	\end{align*}
	The proof is complete. \qed
\end{proof}

\begin{remark}
	\begin{enumerate}
		\item[{\rm (i)}] \textup{In \cite{attouch2021effect}, for $\gamma$-strongly convex function $f$, Attouch et al. considered the system \eqref{system2} with the constant damping  $\alpha =2\sqrt{\gamma}$, and demonstrated that, depending on the formulation, different integrability conditions imposed on the error term $\epsilon(t)$ are sufficient to guarantee the convergence rates of the system \eqref{system2}. Therefore, Theorem \ref{3.2} can be viewed as an extension of \cite[Theorem 4.2]{attouch2021effect}, extending the scope from strongly convex to strongly quasiconvex functions.}
		\item[{\rm (ii)}] \textup{In Theorems \ref{lianxu} and \ref{3.2}, if $f$ has an $L$-Lipschitz continuous gradient,   Assumptions \ref{assume} are no longer necessary, see \cite[Page 16]{lara2024characterizations} for more details.}
	\end{enumerate}
\end{remark}

\section{The inertial accelerated algorithm}

In this section,   under the assumption that $\nabla f$ is $L$-Lipschitz continuous, we investigate the asymptotic properties of the inertial accelerated algorithms obtained by a natural explicit temporal discretization of the systems \eqref{system1} and \eqref{system2}.

Firstly, we analyze the inertial accelerated algorithm derived from the system \eqref{system1}. Consider the constant stepsize $\sqrt{s}$ and set $t_{k} = k\sqrt{s}, ~x_{k} = x(t_{k})$. Then,
\begin{align*}
	\frac{x_{k+1}-2x_{k} + x_{k-1}}{s} +  \alpha \frac{x_{k}-x_{k-1}}{\sqrt{s}} + \nabla f \left( x_{k}+\beta \frac{x_{k}-x_{k-1}}{\sqrt{s}} \right) = 0.
\end{align*}

\noindent Equivalently, it can be written as
\begin{align} \label{suanfa1}
	x_{k+1} = x_{k} + (1-\alpha \sqrt{s})(x_{k}-x_{k-1}) - s \nabla f \left( x_{k}+\frac{\beta}{\sqrt{s}}(x_{k}-x_{k-1}) \right).
\end{align}
  Denoting the constant $1-\alpha \sqrt{s}$ and $\frac{\beta}{\sqrt{s}}$ by $\alpha$ and $\beta$, respectively. Solving \eqref{suanfa1} with respect to $x_{k+1}$ gives the following inertial accelerated algorithm:
\begin{align} \label{algo}
	\left\{
	\begin{array}{lll}
		  y_{k} = x_{k} + \alpha(x_{k}-x_{k-1}), \\ [2mm]
		z_{k} = x_{k} + \beta(x_{k}-x_{k-1}),\\ [2mm]
		x_{k+1} = y_{k} - s \nabla f(z_{k}).
	\end{array}
	\right.
	\tag{IAA}
\end{align}

In the following, we analyze the convergence behaviors of the algorithm \eqref{algo}.

\begin{theorem} \label{lisan}
	Let $f: \mathbb{R}^n \rightarrow \mathbb{R}$ be a $\gamma$-strongly quasiconvex function with $L$-Lipschitz continuous gradient and let $\{x_k\}$ be the sequence generated by \eqref{algo}. Suppose that $s = \frac{1}{L}$,
	  $\alpha \in \left(0, \frac{1}{2} \right)$ and
	$\frac{1+\alpha^2-\sqrt{-15\alpha^4 +2\alpha^2 + 1}}{8\alpha} < \beta < \min \left\{ \alpha,~ \frac{1+\alpha^2+\sqrt{-15\alpha^4 +2\alpha^2 + 1}}{8\alpha} \right\}$. Then,  for $x^{*} = {\rm argmin}_{\mathbb{R}^n} f$, it holds that
	\begin{align*}
		\left\{
		\begin{array}{lll}
			f(x_{k}) - f(x^{*}) \leq \mathcal{E}_1(1 - \rho)^{k-1},\\ [2mm]
			\|x_k - x^{*}\|^2 \leq \frac{4\mathcal{E}_1}{\gamma}(1 - \rho)^{k-1},\\ [2mm]
			\|x_k - x_{k-1}\|^2 \leq \frac{2\alpha \mathcal{E}_1}{L\beta}(1 - \rho)^{k-1},
		\end{array}
		\right.
	\end{align*}
	where  $\mathcal{E}_1:= f(x_1) - f(x^{*}) + \frac{L\beta}{2\alpha} \|x_1 - x_0\|^2$ and $\rho := \min \left\{\frac{\frac{1}{2L}\left(1-\frac{\beta}{\alpha} \right)}{\frac{2L}{\gamma^2} + \frac{\beta}{2}}, \frac{\frac{L}{2\alpha}\big((\alpha^2+1)\beta- 4\alpha\beta^2 - \alpha^3 \big)}{\frac{\beta}{2}\left(1+L\beta + \frac{L}{\alpha}\right)} \right\}$.
\end{theorem}

\begin{proof}
	Given $x^{*} = {\rm argmin}_{\mathbb{R}^n} f$ and $c=\frac{\beta}{\alpha s}$. For all $k \geq 1$, we introduce the   energy sequence
	\begin{align} \label{lsfunc}
		\mathcal{E}_{k} := f(x_{k}) - f(x^{*}) + \frac{c}{2} \|x_{k} - x_{k-1}\|^2.
	\end{align}
Clearly,
	\begin{align} \label{e-e}
		\mathcal{E}_{k+1}-\mathcal{E}_{k} = f(x_{k+1})- f(x_{k}) + \frac{c}{2}(\|x_{k+1} - x_{k}\|^2- \|x_{k} - x_{k-1}\|^2).
	\end{align}
	According to \eqref{algo}, we deduce that
	\begin{align} \label{x-x1}
		x_{k+1} - x_{k} = \alpha(x_{k} - x_{k-1}) - s\nabla f(z_{k}),
	\end{align}
	which follows that
	\begin{align*}
		\|x_{k+1} - x_{k}\|^2 = \alpha^2 \|x_{k} - x_{k-1}\|^2 + s^2\|\nabla f(z_{k})\|^2 - 2\alpha s\langle \nabla f(z_{k}),x_{k} - x_{k-1} \rangle.
	\end{align*}
	This together with \eqref{e-e} yields
	\begin{equation} \label{@@}
		\begin{split}
			\mathcal{E}_{k+1}-\mathcal{E}_{k}
			= & ~  f(x_{k+1}) - f(x_{k}) + \frac{c}{2} \Big((\alpha^2-1) \|x_{k} - x_{k-1}\|^2 \\
			& + s^2\|\nabla f(z_{k})\|^2
			- 2\alpha s \langle \nabla f(z_{k}),x_{k} - x_{k-1} \rangle \Big).
		\end{split}
	\end{equation}
	Moreover, we deduce from \eqref{L} that
	\begin{align}
		f(x_{k+1}) \leq f(z_{k}) + \langle \nabla 	f(z_{k}), x_{k+1} - z_{k} \rangle + \frac{L}{2}\|x_{k+1}-z_{k}\|^2, \label{lk+1}
	\end{align}
	and
	\begin{equation} \label{LL}
		\begin{split}
			f(z_k) & \leq f(x_k) + \langle \nabla f(x_k), z_k-x_k \rangle + \frac{L}{2} \|x_k - z_k\|^2 \\
			& \leq f(x_k) + \langle \nabla f(z_k), z_k-x_k \rangle + \frac{3L}{2} \|x_k - z_k\|^2,
		\end{split}
	\end{equation}
	where the second inequality of \eqref{LL} holds due to
	\[\langle \nabla f(x_k)- \nabla f(z_k), z_k-x_k \rangle \leq \|\nabla f(x_k) - \nabla f(z_k)\| \|z_k-x_k\| \leq L\|z_k-x_k\|^2.
	\]
	
	\noindent Combining \eqref{lk+1} and \eqref{LL}, we have
	\begin{align}\label{&}
		f(x_{k+1}) - f(x_{k}) \leq \langle \nabla f(z_{k}), x_{k+1} - x_{k} \rangle + \frac{L}{2}\|x_{k+1}-z_{k}\|^2 + \frac{3L}{2}\|x_{k}-z_{k}\|^2.
	\end{align}
	By \eqref{algo}, we have
	\begin{equation} \label{x-z}
		\begin{split}
			\|x_{k+1} - z_{k}\|^2 = & ~  \|(\alpha -\beta)(x_k-x_{k-1})-s\nabla f(z_k)\|^2\\
			= & ~  (\alpha-\beta)^2\|x_{k} - x_{k-1}\|^2 + s^2 \|\nabla f(z_{k})\|^2 \\
			& - 2s(\alpha-\beta) \langle \nabla f(z_{k}), x_{k} - x_{k-1} \rangle.
		\end{split}
	\end{equation}
	Together with \eqref{x-x1}, \eqref{x-z} and $z_k - x_k = \beta(x_k-x_{k-1})$, we deduce from \eqref{&} that
	\begin{equation} \label{@}
		\begin{split}
			& f(x_{k+1})-f(x_k) \\
			\leq & \big(\alpha - sL(\alpha-\beta)\big) \langle \nabla f(z_k), x_k-x_{k-1} \rangle + \left(\frac{Ls^2}{2} -s \right) \|\nabla f(z_k)\|^2 \\
			& + \frac{L}{2} \big(3\beta^2+(\alpha-\beta)^2 \big) \|x_k-x_{k-1}\|^2.
		\end{split}
	\end{equation}
	Therefore, combining \eqref{@@} and \eqref{@}, we have
	\begin{align}\label{xishu1}
		\begin{split}
			&\mathcal{E}_{k+1}-\mathcal{E}_{k}\\
			\leq & ~ \big(\alpha-sL(\alpha-\beta)- c\alpha s\big) \langle \nabla f(z_{k}),x_{k} - x_{k-1} \rangle + \left( \frac{Ls^2}{2}+\frac{cs^2}{2}-s \right) \|\nabla f(z_{k})\|^2 \\
			& + \frac{1}{2} \left( 3L\beta^2 + L(\alpha-\beta)^2+c(\alpha^2- 1) \right) \|x_{k} - x_{k-1}\|^2.
		\end{split}
	\end{align}
	
	We now analyze the coefficients of the right side of \eqref{xishu1}. From $s = \frac{1}{L}$ and $ c=\frac{\beta}{\alpha s}$, we obtain
	\begin{align*}
		\alpha-sL(\alpha-\beta)- c\alpha s = 0.
	\end{align*}
	From $0 \leq \beta < \alpha$, $s=\frac{1}{L}$ and $c=\frac{\beta}{\alpha s}$, we have
	  \begin{align}\label{beta00}
		\frac{Ls^2}{2}+\frac{cs^2}{2}-s = \frac{1}{2L}\left(\frac{\beta}{\alpha}-1 \right) < 0.
	\end{align}
	Further, note that
	\begin{equation} \label{beta}
		\begin{split}
			3L\beta^2 + L(\alpha-\beta)^2+c(\alpha^2- 1) = \frac{L}{\alpha}\big(4\alpha\beta^2 - (\alpha^2+1)\beta + \alpha^3 \big).
		\end{split}
	\end{equation}
 According to   $0 < \alpha  < \frac{1}{2}$,  we have
	$
	-15\alpha^4 +2\alpha^2 + 1 > 0.
$
	Then, from $\frac{1+\alpha^2-\sqrt{-15\alpha^4 +2\alpha^2 + 1}}{8\alpha} < \beta < \frac{1+\alpha^2+\sqrt{-15\alpha^4 +2\alpha^2 + 1}}{8\alpha}$, we deduce that \eqref{beta} is negative.  Consequently, by \eqref{xishu1}, \eqref{beta00} and \eqref{beta}, we have 	 \begin{align} \label{rho}
		\mathcal{E}_{k+1}-\mathcal{E}_{k} \leq -\frac{1}{2L}\left(1-\frac{\beta}{\alpha} \right)\|\nabla f(z_{k})\|^2 - \frac{L}{2\alpha}\big((\alpha^2+1)\beta- 4\alpha\beta^2 - \alpha^3 \big)\|x_{k} - x_{k-1}\|^2,
	\end{align}	
where $\frac{1}{2L}\left(1-\frac{\beta}{\alpha} \right)>0$ and $\frac{L}{2\alpha}\big((\alpha^2+1)\beta- 4\alpha\beta^2 - \alpha^3 \big)>0$. 
	
	On the other hand, by \eqref{L}, we have
	\begin{align*}
		f(x_{k}) \leq f(z_{k}) + \langle \nabla f(z_{k}), x_{k} - z_{k} \rangle + \frac{L}{2}\|x_{k}-z_{k}\|^2.
	\end{align*}
	This together with $x_{k} - z_{k} = -\beta(x_{k} - x_{k-1})$ gives
	\begin{align*}
		&f(x_{k})-f(x^{*}) \\ \leq &~ f(z_{k})-f(x^{*}) - \beta \langle \nabla f(z_{k}), x_{k}-x_{k-1} \rangle + \frac{L\beta^2}{2}\|x_{k}-x_{k-1}\|^2 \\
		\leq &~ f(z_{k})-f(x^{*}) + \frac{\beta}{2} \left( \|\nabla f(z_{k})\|^2+\|x_{k}-x_{k-1}\|^2 \right) + \frac{L\beta^2}{2}\|x_{k}-x_{k-1}\|^2\\
		\leq &~ \left( \frac{2L}{\gamma^2} + \frac{\beta}{2} \right) \|\nabla f(z_{k})\|^2+ \left( \frac{\beta}{2} + \frac{L\beta^2}{2}\right) \|x_{k}-x_{k-1}\|^2,
	\end{align*}
	where the last inequality follows from Lemma \ref{PL}. Then, it follows from \eqref{lsfunc} that
  \begin{align} \label{sigma}
		\begin{split}
			\mathcal{E}_{k} \leq &~ \left( \frac{2L}{\gamma^2} + \frac{\beta}{2} \right) \|\nabla f(z_{k})\|^2+ \left( \frac{\beta}{2} + \frac{L\beta^2}{2} + \frac{c}{2} \right) \|x_{k}-x_{k-1}\|^2.
		\end{split}
	\end{align} 
  Let $\rho := \min \left\{\frac{\frac{1}{2L}\left(1-\frac{\beta}{\alpha} \right)}{\frac{2L}{\gamma^2} + \frac{\beta}{2}}, \frac{\frac{L}{2\alpha}\big((\alpha^2+1)\beta- 4\alpha\beta^2 - \alpha^3 \big)}{\frac{\beta}{2}\left(1+L\beta + \frac{L}{\alpha}\right)} \right\}$. Note that $\frac{L}{2\alpha}\big((\alpha^2+1)\beta- 4\alpha\beta^2 - \alpha^3 \big)<\frac{\beta}{2}\left(1+L\beta + \frac{L}{\alpha}\right).$  Then, $0<\rho<1.$ According to \eqref{rho} and \eqref{sigma}, it is easy to show  that
	\begin{align*}
		\mathcal{E}_{k+1} \leq (1 - \rho) \mathcal{E}_{k} \leq (1 - \rho)^k \mathcal{E}_{1}.
	\end{align*}
	By \eqref{lsfunc} and Lemma \ref{4}, we have
$$
		\frac{\gamma}{4}\|x_{k}-x^{*}\|^2 \leq f(x_{k}) - f(x^{*}) \leq (1 - \rho)^{k-1} \mathcal{E}_{1} $$
and
		$$\frac{L\beta}{2\alpha}\|x_k - x_{k-1}\|^2 \leq (1 - \rho)^{k-1} \mathcal{E}_{1}.
$$
	i.e., the linear convergence of algorithm \eqref{algo} has been achieved. \qed
\end{proof}

In the sequel, by temporal discretization of the system \eqref{system2}, we propose the following inertial accelerated algorithm:
\begin{align} \label{algo-error}
	\left\{
	\begin{array}{lll}
		\textcolor[HTML]{00A36C}{ y_{k} = x_{k} + \alpha(x_{k}-x_{k-1})}, \\ [2mm]
		z_{k} = x_{k} + \beta(x_{k}-x_{k-1}),\\ [2mm]
		x_{k+1} = y_{k} - s \nabla f(z_{k}) + s \epsilon_{k}.
	\end{array}
	\right.
	\tag{IAA-Per}
\end{align}
\begin{remark}
 \textup{ Similarly to the algorithm \eqref{algo-error},  we can  introduce several other algorithms in terms of the temporal discretization of the perturbed version of the systems \eqref{chzh} and \eqref{hess}. These algorithms include    (\ref{heavyball}) \cite{polyak1964some} with perturbations (HBM-Per), (\ref{heavyballhess})   \cite{hadjisavvas2025heavy} with perturbations (HBM-H-Per), (\ref{NAG})  \cite{nesterov1983method} with   perturbations (NAG-Per), and (\ref{NAGhess})  \cite{hadjisavvas2025heavy} with   perturbations (NAG-H-Per). Further, we will  examine the performance of these algorithms  through   numerical experiments in Section 5. }
\end{remark}

Now, we analyze the convergence properties of our algorithm \eqref{algo-error}.
\begin{theorem} \label{4.2}
	Let $f: \mathbb{R}^n \rightarrow \mathbb{R}$ be a $\gamma$-strongly quasiconvex function with $L$-Lipschitz gradient and let $\{x_k\}$ be the sequence generated by \eqref{algo-error}. Suppose that $s = \frac{1}{L}$,
	$\alpha \in \left( 0, \frac{1}{2} \right)$,
	$~\frac{1-\sqrt{1-16\alpha^4}}{8\alpha} < \beta < \min \left\{ \frac{\alpha}{2}, \frac{1 + \sqrt{1-16\alpha^4}}{8\alpha} \right\}$
	and for some $p>0$, $\|\epsilon_k\|= \mathcal{O} \left( \frac{1}{k^{p}} \right)$, as $k \rightarrow +\infty$. Then, 	for $x^{*} = {\rm argmin}_{\mathbb{R}^n} f$, it holds that
	\begin{align*}
		\left\{
		\begin{array}{lll}
			f(x_k) - f(x^{*}) = \mathcal{O} \left( \frac{1}{k^{2p}} \right), ~\textup{as}~ k \rightarrow +\infty,\\ [2mm]
			\|x_k - x^{*}\| = \mathcal{O} \left( \frac{1}{k^{p}} \right), ~\textup{as}~ k \rightarrow +\infty,\\ [2mm]
			\|x_k - x_{k-1}\| = \mathcal{O} \left( \frac{1}{k^{p}} \right), ~\textup{as}~ k \rightarrow +\infty.
		\end{array}
		\right.
	\end{align*}

\end{theorem}

\begin{proof}
	We consider the energy sequence as in \eqref{lsfunc}, recalling that
	\begin{align} \label{*2}
		\mathcal{E}_{k+1}-\mathcal{E}_{k}
		= f(x_{k+1}) - f(x_{k}) + \frac{c}{2} (\|x_{k+1} - x_{k}\|^2 - \|x_{k} - x_{k-1}\|^2).
	\end{align}
	According to \eqref{algo-error}, we deduce that
	\begin{align} \label{xk-k}
		x_{k+1} - x_{k} = \alpha(x_{k} - x_{k-1}) - s\nabla f(z_{k}) + s\epsilon_{k},
	\end{align}
	which follows that
	\begin{align*}
		\|x_{k+1} - x_{k}\|^2 = & ~ \alpha^2 \|x_{k} - x_{k-1}\|^2 + s^2\|\nabla f(z_{k})\|^2 + s^2 \|\epsilon_{k}\|^2 \\
		& - 2\alpha s \langle \nabla f(z_{k}),x_{k} - x_{k-1} \rangle \\
		& + 2\alpha s\langle x_{k} - x_{k-1},\epsilon_{k} \rangle - 2s^2 \langle \nabla f(z_{k}),\epsilon_{k} \rangle.
	\end{align*}
	This together with \eqref{*2} yields
	\begin{equation}
		\begin{split} \label{k-k2}
			\mathcal{E}_{k+1}-\mathcal{E}_{k} = & ~  f(x_{k+1}) - f(x_{k}) + \frac{c}{2} \Big((\alpha^2 - 1) \|x_{k} - x_{k-1}\|^2 \\
			& + s^2\|\nabla f(z_{k})\|^2 + s^2 \|\epsilon_{k}\|^2 - 2\alpha s \langle \nabla f(z_{k}),x_{k} - x_{k-1} \rangle \\
			& + 2\alpha s \langle x_{k} - x_{k-1},\epsilon_{k} \rangle - 2s^2 \langle \nabla f(z_{k}),\epsilon_{k} \rangle \Big).
		\end{split}
	\end{equation}
	By \eqref{algo-error}, we have
	\begin{equation*}
		\begin{split}
			\|x_{k+1} - z_{k}\|^2 = & ~  (\alpha-\beta)^2\|x_{k} - x_{k-1}\|^2 + s^2 \|\nabla f(z_{k})\|^2 + s^2 \|\epsilon_k\|^2 \\
			& - 2s(\alpha-\beta)\langle \nabla f(z_{k}), x_{k} - x_{k-1} \rangle\\
			& + 2s(\alpha -\beta)\langle x_{k} - x_{k-1}, \epsilon_k \rangle - 2s^2 \langle \nabla f(z_{k}), \epsilon_k \rangle.
		\end{split}
	\end{equation*}
	Combining this with \eqref{xk-k} and $z_k - x_k = \beta(x_k - x_{k-1})$, we deduce from \eqref{&} that
	\begin{equation} \label{@11}
		\begin{split}
			&f(x_{k+1})-f(x_k)\\ \leq & ~ \big(\alpha - sL(\alpha-\beta)\big) \langle \nabla f(z_k), x_k-x_{k-1} \rangle
			+ \left(\frac{Ls^2}{2} -s \right) \|\nabla f(z_k)\|^2 \\
			& + (s-Ls^2) \langle \nabla f(z_k), \epsilon_k \rangle
			+ \left(\frac{3L}{2}\beta^2+\frac{L}{2}(\alpha-\beta)^2 \right) \|x_k-x_{k-1}\|^2 \\
			& + \frac{Ls^2}{2} \|\epsilon_k\|^2 + Ls(\alpha-\beta)\langle x_k - x_{k-1}, \epsilon_k \rangle.
		\end{split}
	\end{equation}
	Therefore, it follows from \eqref{k-k2} that
	\begin{align*}
		&\mathcal{E}_{k+1}-\mathcal{E}_{k}\\
		\leq & ~ \big(\alpha- sL(\alpha-\beta)-c\alpha s\big) \langle \nabla f(z_{k}), x_{k} - x_{k-1} \rangle \\
		& +\left(\frac{Ls^2}{2} + \frac{cs^2}{2}-s\right)\|\nabla f(z_{k})\|^2 + (s-Ls^2-cs^2)\langle \nabla f(z_{k}), \epsilon_{k} \rangle \\
		& + \left(\frac{3L}{2}\beta^2 + \frac{L}{2}(\alpha-\beta)^2 + \frac{c}{2}(\alpha^2 - 1)\right)\|x_{k} - x_{k-1}\|^2  \\
		& + \left(\frac{Ls^2}{2} + \frac{cs^2}{2}\right) \|\epsilon_{k}\|^2 + \big(Ls(\alpha-\beta) +c\alpha s\big) \langle x_{k} - x_{k-1},\epsilon_{k} \rangle.
	\end{align*}
	From $s = \frac{1}{L}$ and $ c=\frac{\beta}{\alpha s}$, we have
	\begin{align*}
		\left\{
		\begin{array}{lll}
			\alpha- sL(\alpha-\beta)-c\alpha s=0,~~ &	\dfrac{Ls^2}{2} + \dfrac{cs^2}{2}-s = \dfrac{cs^2}{2}-\dfrac{s}{2}, \\ [2mm] s-Ls^2-cs^2 = -cs^2, &
			Ls(\alpha-\beta) +c\alpha s=\alpha.
		\end{array}
		\right.
	\end{align*}
	Thus,
	\begin{equation} \label{**2}
		\begin{split}
			&\mathcal{E}_{k+1}-\mathcal{E}_{k}\\
\leq& \left(\frac{cs^2}{2} - \frac{s}{2} \right)\|\nabla f(z_{k})\|^2 -cs^2\langle \nabla f(z_{k}), \epsilon_{k} \rangle \\
			&~ + \left(\frac{3L}{2}\beta^2 + \frac{L}{2}(\alpha-\beta)^2 + \frac{c}{2}(\alpha^2 - 1)\right)\|x_{k} - x_{k-1}\|^2  \\
			&~ + \left(\frac{Ls^2}{2} + \frac{cs^2}{2}\right) \|\epsilon_{k}\|^2 + \alpha \langle x_{k} - x_{k-1},\epsilon_{k} \rangle \\
			\leq& \left(cs^2 - \frac{s}{2} \right)\|\nabla f(z_{k})\|^2 + \left( \frac{Ls^2}{2} + cs^2 + \frac{\alpha}{2L\beta} \right) \|\epsilon_{k}\|^2\\
			& ~+ \left(\frac{3L}{2}\beta^2 + \frac{L}{2}(\alpha-\beta)^2 + \frac{c}{2}(\alpha^2 - 1)+\frac{L}{2}\alpha \beta \right)\|x_{k} - x_{k-1}\|^2,
		\end{split}
	\end{equation}
	where the second inequality holds due to
$
		-cs^2\langle \nabla f(z_{k}), \epsilon_{k} \rangle \leq \frac{cs^2}{2} \|\nabla f(z_{k})\|^2 + \frac{cs^2}{2} \|\epsilon_{k}\|^2
$
	and
$
		 \alpha  \langle x_{k} - x_{k-1}, \epsilon_{k} \rangle \leq \frac{L}{2}\alpha \beta \|x_{k} - x_{k-1}\|^2 + \frac{\alpha}{2L\beta} \|\epsilon_{k}\|^2.
$
	
	We now analyze the coefficients of the right side of \eqref{**2}. From $c=\frac{\beta}{\alpha s}$ and $0 \leq \beta < \frac{\alpha}{2}$, we have
	\begin{align*}
		cs^2 - \frac{s}{2} = \frac{1}{L} \left(\frac{\beta}{\alpha}-\frac{1}{2} \right) < 0.
	\end{align*}
	Further, note that
	\begin{equation} \label{@2}
		\begin{split}
			\frac{3L}{2}\beta^2 + \frac{L}{2}(\alpha-\beta)^2 + \frac{c}{2}(\alpha^2 - 1)+\frac{L}{2}\alpha \beta =\frac{L}{2\alpha}\left(4\alpha\beta^2 - \beta + \alpha^3 \right).
		\end{split}
	\end{equation}
	From $\alpha \in \left( 0, \frac{1}{2}\right)$, we have
$ 1-16\alpha^4 > 0.
$
	Then, from $\frac{1-\sqrt{1-16\alpha^4 }}{8\alpha} < \beta < \frac{1 + \sqrt{1-16\alpha^4 }}{8\alpha}$, we deduce that $4\alpha\beta^2 - \beta + \alpha^3 < 0$, which means that \eqref{@2} is negative. Consequently, we obtain that
	\begin{equation}\label{rho2}\begin{split}
		&\mathcal{E}_{k+1}-\mathcal{E}_{k}\\
 \leq &- \frac{1}{L} \left(\frac{1}{2} - \frac{\beta}{\alpha} \right) \|\nabla f(z_{k})\|^2 - \frac{L}{2\alpha}\left(\beta - 4\alpha\beta^2 - \alpha^3 \right) \|x_{k} - x_{k-1}\|^2 + N\|\epsilon_{k}\|^2,
	\end{split}
\end{equation}
	where $ \frac{1}{L} \left(\frac{1}{2} - \frac{\beta}{\alpha} \right) > 0$ , $\frac{L}{2\alpha}\left(\beta - 4\alpha\beta^2 - \alpha^3 \right) > 0$ and
	$N := \frac{1}{L} \left( \frac{1}{2} + \frac{\beta}{\alpha} + \frac{\alpha}{2\beta} \right)$. 
	
 Let $\sigma := \min \left\{\frac{\frac{1}{L} \left(\frac{1}{2} - \frac{\beta}{\alpha} \right)}{\frac{2L}{\gamma^2} + \frac{\beta}{2}}, \frac{\frac{L}{2\alpha}\left(\beta - 4\alpha\beta^2 - \alpha^3 \right)}{\frac{\beta}{2}\left(1+L\beta + \frac{L}{\alpha}\right)} \right\}$.  Note that $\frac{L}{2\alpha}\big( \beta- 4\alpha\beta^2 - \alpha^3 \big)<\frac{\beta}{2}\left(1+L\beta + \frac{L}{\alpha}\right).$  Then, $0<\sigma<1.$ Together with \eqref{sigma} and \eqref{rho2}, we get
	\begin{align*}
		\mathcal{E}_{k+1} \leq   ~ (1 - \sigma) \mathcal{E}_{k} + N \| \epsilon_{k} \|^2
		\leq  ~ (1 - \sigma)^k \mathcal{E}_{1} + N \sum_{i=1}^{k}(1 - \sigma)^{k-i} \| \epsilon_i\|^2.
	\end{align*} 
	Clearly, according to Lemma \ref{jishu}, we obtain
$
		\mathcal{E}_k = {\mathcal{O}}\left( \frac{1}{k^{2p}} \right), ~\textup{as}~ t \rightarrow +\infty.
$
	It follows that
	\begin{align*}
		\left\{
		\begin{array}{lll}
			f(x_k) - f(x^{*}) = \mathcal{O} \left( \frac{1}{k^{2p}} \right), ~\textup{as}~ k \rightarrow +\infty,\\ [2mm]
			\|x_k - x^{*}\| = \mathcal{O} \left( \frac{1}{k^{p}} \right), ~\textup{as}~ k \rightarrow +\infty,\\ [2mm]
			\|x_k - x_{k-1}\| = \mathcal{O} \left( \frac{1}{k^{p}} \right), ~\textup{as}~ k \rightarrow +\infty.
		\end{array}
		\right.
	\end{align*}
	The proof is complete. \qed
\end{proof}

\begin{remark}
	In \cite[Section 4.2]{attouch2021effect}, for $\gamma$-strongly convex function $f$, Attouch et al.   considered the influence of perturbations on dynamical systems with Hessian-driven damping. However, their investigations were restricted to the convex setting and did not extend to algorithm design. Therefore, the results of Theorem \ref{4.2} appear new in the context of strongly quasiconvex functions.
\end{remark}
\section{Numerical experiments}

In this section, motivated by the examples reported in \cite[Section 5]{hadjisavvas2025heavy} and \cite[Section 6]{lara2024characterizations}, we give some numerical experiments to validate the proposed inertial accelerated algorithm \eqref{algo} and its perturbed version \eqref{algo-error}.

Firstly, to validate the effectiveness of implicit Hessian-driven damping, we compare our algorithm \eqref{algo} with the Heavy Ball method \eqref{heavyball} introduced in \cite{polyak1964some}, the Nesterov's accelerated gradient method \eqref{NAG} introduced in \cite{nesterov1983method}, the Heavy Ball acceleration method with Hessian correction \eqref{heavyballhess} and the Nesterov's accelerated gradient method with Hessian correction \eqref{NAGhess} introduced in \cite{hadjisavvas2025heavy}.
\begin{example}\cite[Example 30]{lara2024characterizations} Consider the following nonconvex optimization problem:
	\begin{align*}
		\min_{x \in \mathbb{R}} f(x) = x^2 + 2\sin^{2}x.
	\end{align*}
	Obviously,  $f$ is a $\gamma$-strongly quasiconvex function with modulus $\gamma=\frac{1}{2}$ and $\nabla f(x) = 2x + 2\sin 2x $ is $L$-Lipschitz continuous with modulus $L=6$. It is easy to know that $x^{*}=0$ and the optimal value $f(x^{*})=0$.
	
	For the algorithms (\ref{heavyball}), (\ref{heavyballhess}), (\ref{NAG}), and (\ref{NAGhess}), we choose   $\alpha=0.7$, $\theta=0.05$ and $\beta=\frac{1}{4L}$. For our algorithm \eqref{algo}, we choose  $\alpha=0.3$, $\beta=0.2$  and $s=\frac{1}{L}=\frac{1}{6}$. The initial points for all algorithms are set to $x_0=x_1=3$.

	Under the above parameters settings, we give the following numerical experiment for these algorithms.

	Figure \ref{fig:12}  displays  the evolution of the objective function value error  $f(x_k)-f(x^{*})$ and iteration error  $\|x_k-x^{*}\|$  under a stopping tolerance of $tol = 10^{-10}$ across different algorithms.

	As illustrated in Figure \ref{fig:12}, we can observe that Algorithms \eqref{heavyballhess} and \eqref{NAGhess}, which are enhanced with explicit Hessian-driven damping, outperform the classical \eqref{heavyball} and \eqref{NAG} in terms of either the convergence rate or  the reduction in oscillations. This trend is further amplified in our algorithm \eqref{algo}, which converges to the optimum faster and exhibits less oscillations.

	\begin{figure}[htbp]
		\centering
		\begin{minipage}{0.45\textwidth}
			\centering
			\includegraphics[width=\linewidth]{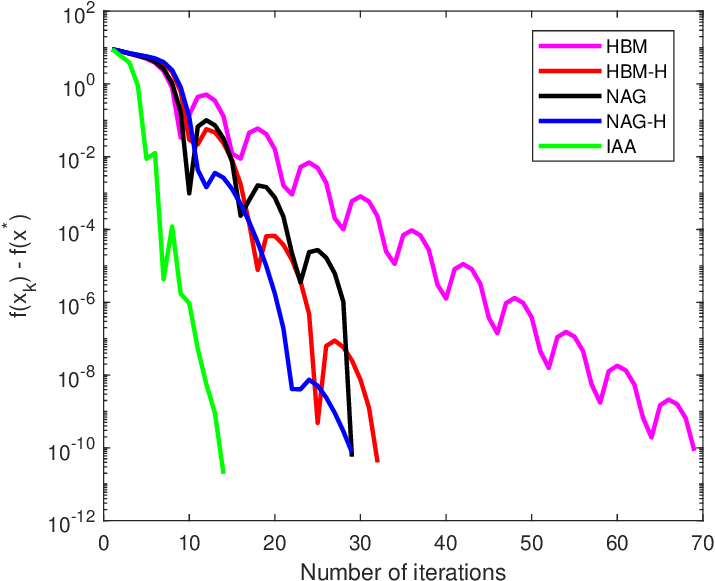}
		\end{minipage}
		\hfill
		\begin{minipage}{0.45\textwidth}
			\centering
			\includegraphics[width=\linewidth]{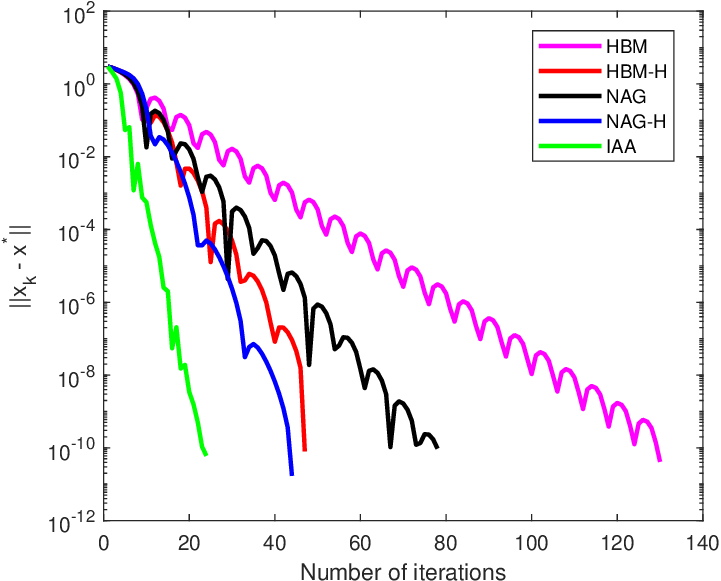}
		\end{minipage}
		\caption{Objective function value error and iteration error for different algorithms.}
		\label{fig:12}
	\end{figure}
	
Figure \ref{fig:34} displays the evolution of successive error  $\|x_k-x_{k-1}\|$ under the stopping tolerance $tol = 10^{-10}$ and iteration trajectory $x_k$ with a maximum iteration limit of $\max\_ iter = 50$ across different algorithms.

	\begin{figure}[htbp]
		\centering
		\begin{minipage}{0.45\textwidth}
			\centering
			\includegraphics[width=\linewidth]{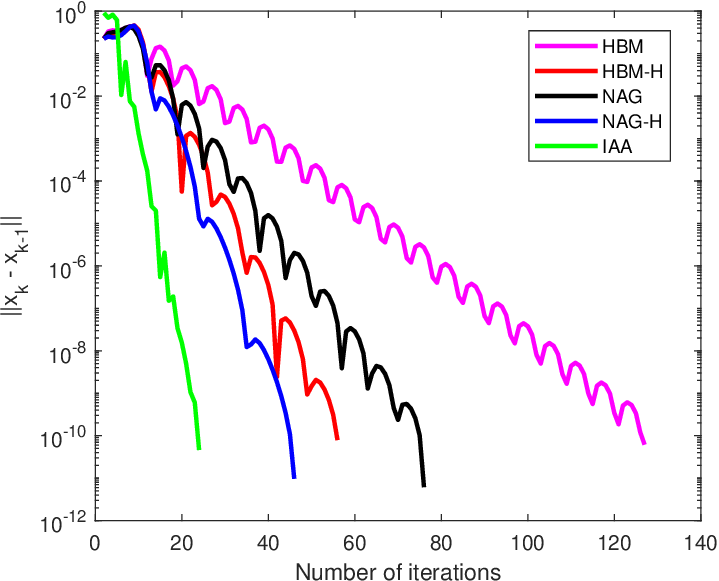}
		\end{minipage}
		\hfill
		\begin{minipage}{0.45\textwidth}
			\centering
			\includegraphics[width=\linewidth]{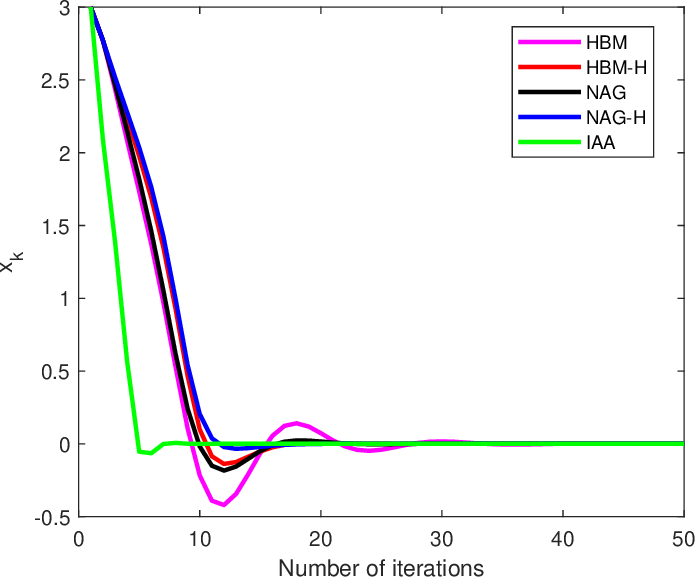}
		\end{minipage}
		\caption{Successive errors and iteration trajectories  for different algorithms.}
		\label{fig:34}
	\end{figure}

Figure \ref{fig:34} further confirms that our algorithm \eqref{algo} converges to the optimum  faster and demonstrates less oscillation.
\end{example}

In the following example, we examine the performance of the algorithms under external perturbations,   comparing our algorithm \eqref{algo-error} with (HBM-Per),  (HBM-H-Per),  (NAG-Per), and   (NAG-H-Per).
\begin{example}
	Consider the following nonconvex optimization problem:
	\begin{align*}
		\min_{(x,y) \in \mathbb{R}^2} f(x,y)=\frac{1}{10} x^2 + \frac{1}{5} y^2 - \arctan\frac{1}{x^2 + 2y^2 + 0.2}.
	\end{align*}
	Clearly, $f$ is a twice differentiable and $\gamma$-strongly quasiconvex function with modulus $\gamma = 0.2$. The 3D graph of the function $f$ is shown in Figure \ref{fig:3d}. Thus, $f$ is a nonconvex function. Moreover, it is easy to show that $\nabla f$ is Lipschitz continuous, $(x^{*}, y^{*})=(0, 0)$ and the optimal value $f(x^{*},y^*) = -\arctan 5  \approx -1.373401$.
	\begin{figure}[htbp]
		\centering
		\begin{minipage}{0.45\textwidth}
			\centering
			\includegraphics[width=\linewidth]{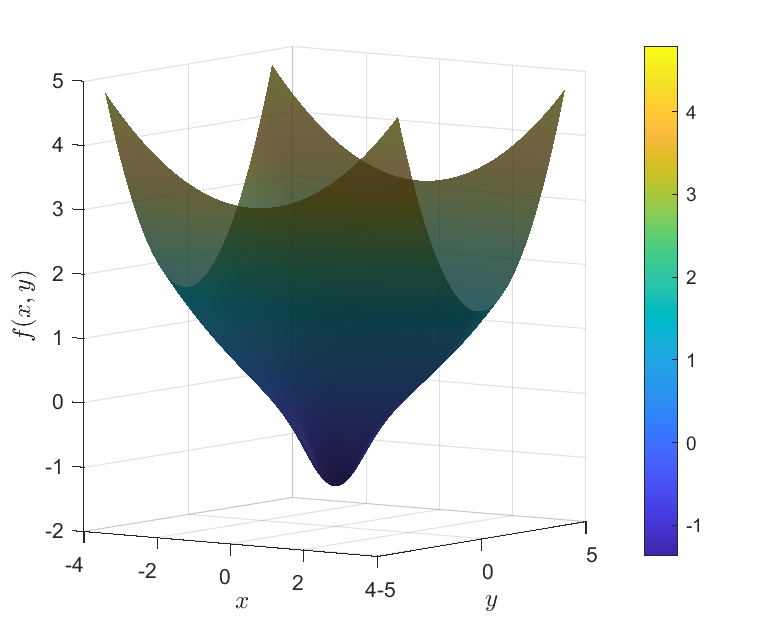}
		\end{minipage}
		\caption{The 3D graph of the function $f$.}
		\label{fig:3d}
	\end{figure}
	
	For the algorithms (HBM-Per), (HBM-H-Per), (NAG-Per) and (NAG-H-Per), we set the parameters as   $\alpha = 0.7$, $\theta = 0.05$, and $\beta = 0.04$. For our algorithm \eqref{algo-error}, the parameters are chosen as  $\alpha = 0.4$, $\beta = 0.15$, and $s = 0.125$. In all algorithms, we set $(x_0,y_0) = (x_1,y_1) = (3,3)$ and noise $\epsilon_k \sim \mathcal{N}(0, \sigma_k^2 I_2)$ with decaying standard deviation $\sigma_k = \frac{0.001}{1+0.01k}$. The stopping criterion is defined as reaching a maximum of $200$ iterations.
	
	Based on the above parameter settings, we carry out the following numerical experiment involving these algorithms.
	
	Figures \ref{fig:1222} and \ref{fig:3422} analyze  the evolution of the objective function value error $f(x_k, y_k)-f(x^{*}, y^{*})$, iteration error  $\|(x_k, y_k)-(x^{*}, y^{*})\|$, gradient norm $\|\nabla f(x_k,y_k)\|$ and iteration trajectory $(x_k,y_k)$.
	
	\begin{figure}[htbp]
		\centering
		\begin{minipage}{0.45\textwidth}
			\centering
			\includegraphics[width=\linewidth]{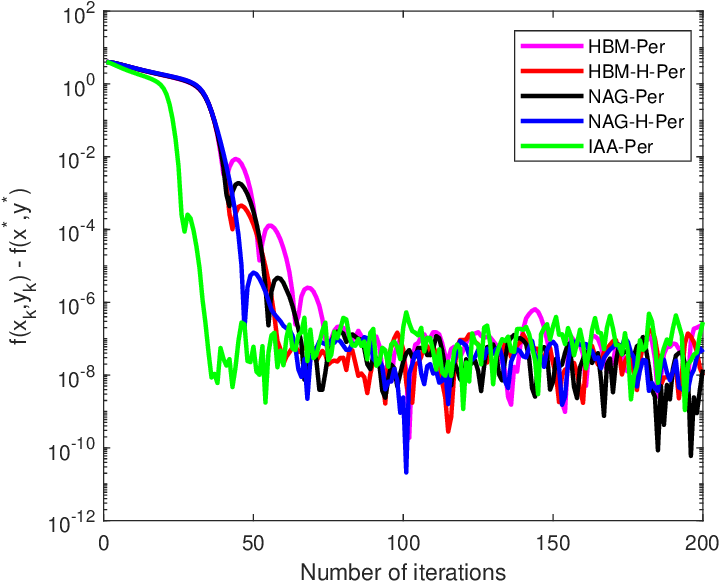}
		\end{minipage}
		\hfill
		\begin{minipage}{0.45\textwidth}
			\centering
			\includegraphics[width=\linewidth]{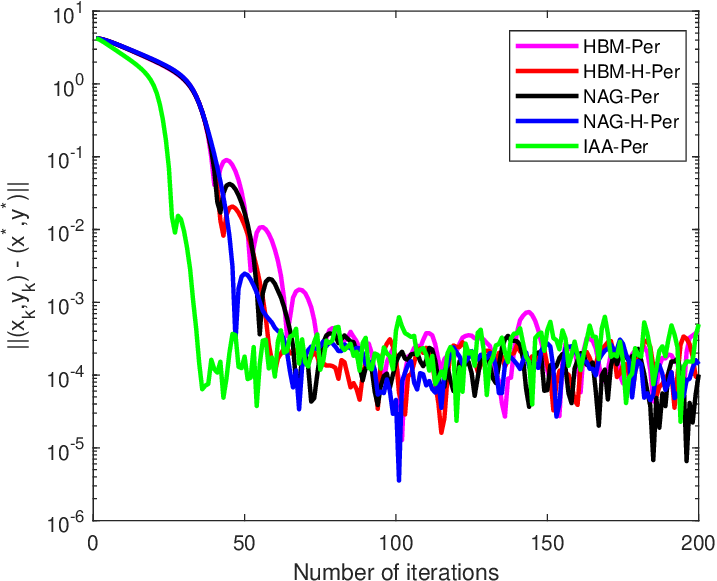}
		\end{minipage}
		\caption{Objective function value error and iteration error for different algorithms.}
		\label{fig:1222}
	\end{figure}
	
	\begin{figure}[htbp]
		\centering
		\begin{minipage}{0.45\textwidth}
			\centering
			\includegraphics[width=\linewidth]{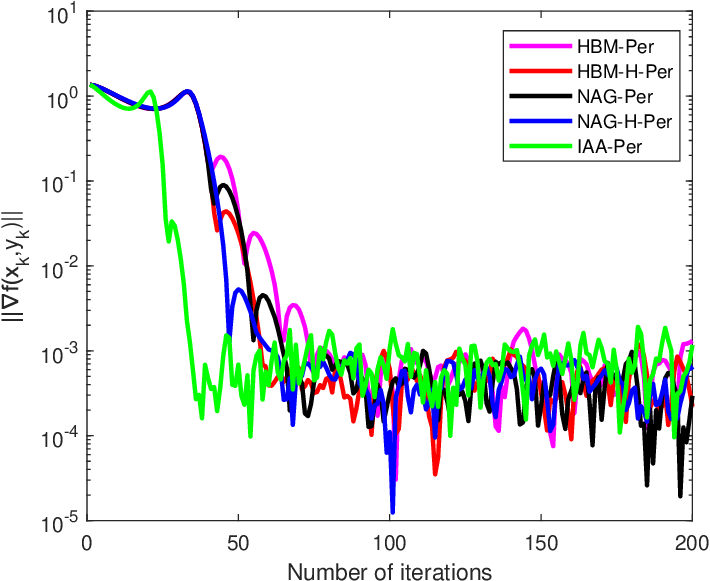}
		\end{minipage}
		\hfill
		\begin{minipage}{0.45\textwidth}
			\centering
			\includegraphics[width=\linewidth]{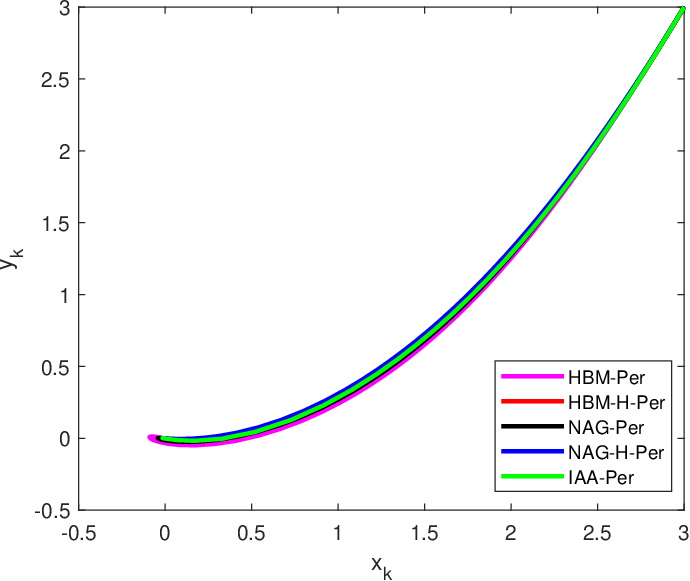}
		\end{minipage}
		\caption{Gradient norm and iteration trajectory for different algorithms.}
		\label{fig:3422}
	\end{figure}
	
	As shown in Figures \ref{fig:1222} and \ref{fig:3422}, our algorithm \eqref{algo-error} outperforms the other four algorithms in reducing oscillations.
	
\end{example}

\section{Conclusions}
In this paper, for strongly quasiconvex function, we establish exponential convergence rates for the objective function value error, the trajectory and the velocity vector along the trajectory generated by the inertial dynamical system with implicit Hessian-driven damping. We also investigate the dynamical system with exogenous additive perturbations and obtain its convergence results. Applying natural explicit discretization to the continuous dynamical system, we develop the corresponding inertial accelerated algorithm and its perturbed version. Through numerical experiments, it is observed that our algorithms achieve faster convergence and exhibit less oscillations.

Although we present some new results in applying dynamical system approaches to strongly quasiconvex optimization problems, there are some remaining questions to be investigated in the future. For example, whether dynamical systems with asymptotic vanishing damping can be employed to solve such problems. And it is also important to consider how these methods can be extended to handle nonconvex optimization problems with linear equality constraints.

\section*{Funding}
{\small This research is supported by the Natural Science Foundation of Chongqing (CSTB2024NSCQ-MSX0651) and the Team Building Project for Graduate Tutors in Chongqing (yds223010).}

\section*{Data availability}
{\small The authors confirm that all data generated or analysed during this study are included in this article.}
\section*{Declaration}
{\small $\mathbf{Conflict ~of~ interest}$ No potential conflict of interest was reported by the authors.}

\bibliographystyle{plain}

\end{document}